\documentclass[a4paper,12pt,final]{amsart}
\usepackage{times,a4wide,mathrsfs,amssymb,amsmath,amsthm,wasysym}

\newcommand{\C}{\mathbb{C}}
\newcommand{\ZZ}{\mathbb{Z}}

\newcommand{\QQ}{\mathbb{Q}}
\newcommand{\NN}{\mathbb{N}}
\newcommand{\PP}{\mathbb{P}}

\newcommand{\A}{\mathbb{A}}
\newcommand{\OO}{\mathcal O}
\newcommand{\Ss}{\mathcal S}

\newcommand{\oo}{\mathfrak o}

\newcommand{\DD}{\mathcal D}
\newcommand{\XX}{\mathcal X}

\newcommand{\MM}{\mathcal M}

\newcommand{\gr}{\hbox{Gr}}

\newcommand{\rom}{\romannumeral}

\newcommand\undermat[2]{
  \makebox[0pt][l]{$\smash{\underbrace{\phantom{
    \begin{matrix}#2\end{matrix}}}_{\text{$#1$}}}$}#2}
    
 \makeatletter
\newcommand*{\da@rightarrow}{\mathchar"0\hexnumber@\symAMSa 4B }
\newcommand*{\da@leftarrow}{\mathchar"0\hexnumber@\symAMSa 4C }
\newcommand*{\xdashrightarrow}[2][]{%
  \mathrel{%
    \mathpalette{\da@xarrow{#1}{#2}{}\da@rightarrow{\,}{}}{}%
  }%
}
\newcommand{\xdashleftarrow}[2][]{%
  \mathrel{%
    \mathpalette{\da@xarrow{#1}{#2}\da@leftarrow{}{}{\,}}{}%
  }%
}
\newcommand*{\da@xarrow}[7]{%
  \sbox0{$\ifx#7\scriptstyle\scriptscriptstyle\else\scriptstyle\fi#5#1#6\m@th$}%
  \sbox2{$\ifx#7\scriptstyle\scriptscriptstyle\else\scriptstyle\fi#5#2#6\m@th$}%
  \sbox4{$#7\dabar@\m@th$}%
  \dimen@=\wd0 %
  \ifdim\wd2 >\dimen@
    \dimen@=\wd2 %
  \fi
  \count@=2 %
  \def\da@bars{\dabar@\dabar@}%
  \@whiledim\count@\wd4<\dimen@\do{%
    \advance\count@\@ne
    \expandafter\def\expandafter\da@bars\expandafter{%
      \da@bars
      \dabar@ 
    }%
  }%
  \mathrel{#3}%
  \mathrel{%
    \mathop{\da@bars}\limits
    \ifx\\#1\\%
    \else
      _{\copy0}%
    \fi
    \ifx\\#2\\%
    \else
      ^{\copy2}%
    \fi
  }%
  \mathrel{#4}%
}
\makeatother

\DeclareMathOperator{\aut}{Aut}

\DeclareMathOperator{\ide}{id}

\newtheorem{theorem}{Theorem}[section]
\newtheorem{claim}[theorem]{Claim}
\newtheorem{lemma}[theorem]{Lemma}

\newtheorem{corollary}[theorem]{Corollary}
\newtheorem{proposition}[theorem]{Proposition}

\newtheorem{remark}[theorem]{Remark}
\newtheorem{definition}[theorem]{Definition}
\newtheorem{convention}{Conventions}
\newtheorem{question}[theorem]{Question}
\newtheorem{notation}[theorem]{Notation}

\newtheorem{nonumbering}{Theorem}

\newtheorem{nonumberingt}{Acknowledgements}

\begin{document}
\author[Robert Laterveer]
{Robert Laterveer}

\address{Institut de Recherche Math\'ematique Avanc\'ee,
CNRS -- Universit\'e 
de Strasbourg,\
7 Rue Ren\'e Des\-car\-tes, 67084 Strasbourg CEDEX,
FRANCE.}
\email{robert.laterveer@math.unistra.fr}

\title{Algebraic cycles and triple $K3$ burgers}

\begin{abstract} We consider surfaces of geometric genus $3$ with the property that their transcendental cohomology splits into $3$ pieces, each piece coming from a $K3$ surface.
For certain families of surfaces with this property, we can show there is a similar splitting on the level of Chow groups (and Chow motives). 

 \end{abstract}


\maketitle

\section{Introduction}

This note is about a class of surfaces which we propose to call {\em triple $K3$ burgers\/}. These are complex smooth projective surfaces $S$ of general type of geometric genus $3$, with the property that there exist $3$ $K3$ surfaces $X_j$ such that the transcendental cohomology $H^2_{tr}(S)$ splits
  \begin{equation}\label{defprop}  H^2_{tr}(S)\cong H^2_{tr}(X_0)\oplus H^2_{tr}(X_1)\oplus H^2_{tr}(X_2)\ .
  \end{equation}
  (The precise definition of triple $K3$ burgers is more restrictive, cf. definition \ref{def}.)
  
   
  The crystal ball of the Bloch--Beilinson--Murre conjectures \cite{J2}, \cite{J4}, \cite{Vo}, \cite{MNP}, \cite{Mur} predicts that relation (\ref{defprop}) also holds on the level of Chow groups (and provided the Hodge conjecture is true, the Chow motive of $S$ should be of abelian type, in the sense of \cite{V3}). The main result of this note provides a verification of this prediction in certain cases:
  
  \begin{nonumbering}[=theorem \ref{main}] Let $S$ be a triple $K3$ burger. Assume that either

\noindent
(\rom1) $K_S^2=2$, or

\noindent
(\rom2) $K_S^2=3$ and the canonical map of $S$ is base point free.


 Then there is an isomorphism (induced by a correspondence)
   \[  A^2_{hom}(S)\ \xrightarrow{\cong}\ A^2_{hom}(X_0)\oplus A^2_{hom}(X_1)\oplus A^2_{hom}(X_2)\ ,\]
   where the $X_j$ are the associated $K3$ surfaces. 
  \end{nonumbering}
  
  (Here $A^2_{hom}()$ denotes the Chow group of $0$--cycles of degree $0$ with rational coefficients.)
   
  In each of the cases of theorem \ref{main}, these surfaces do exist (in case (\rom1), they form a family of dimension at least $6$; in case (\rom2) the moduli dimension is $4$).
   
 It is {\em not\/} a coincidence that the surfaces of theorem \ref{main} lie on or close to the Noether line $K^2=2p_g-4$. Indeed (as is known since the fundamental work of Horikawa \cite{Hor1}, \cite{Hor0}, \cite{Hor}, \cite{Hor3}, \cite{Hor4}), the canonical model of a general type surface on or close to the Noether line admits a neat description as complete intersection in a certain weighted projective space. Thanks to such a description, surfaces as in theorem \ref{main} fit in nicely behaved universal families. Then, one can apply the alchemy of Voisin's method of ``spread'' \cite{V0}, \cite{V1}, \cite{Vo} to transmute the base metal of the homological relation (\ref{defprop}) into the pure gold of a rational equivalence.
  
We also prove (subsection \ref{scan}) that a triple $K3$ burger $S$ as in theorem \ref{main} admits a canonical $0$--cycle $\oo_S\in A^2(S)$, such that there is a splitting
  \[   A^2(S)=\QQ[\oo_S]\oplus A^2_{hom}(S)\ .\]
  The cycle $\oo_S$ has the property that the intersection of certain divisors is proportional to $\oo_S$ (proposition \ref{int}). Another characterization of $\oo_S$ is as follows (proposition \ref{orbit}): for any positive integer $k$, the cycle $k\oo_S$ is the unique degree $k$ $0$--cycle $z$ for which the effective orbit $O_z$ has dimension $\ge k$.
  These results are based on similar results for the canonical $0$--cycle of a $K3$ surface \cite{Huyb}, \cite{BV}, \cite{Vo}, \cite{V21}. 
  
  In a sense, the present note is a sequel to \cite{hor}, which dealt with certain surfaces of geometric genus $p_g=2$. The surfaces $S$ of \cite{hor} are also studied in \cite{Gar} and \cite{PZ}; they have the property that their transcendental cohomology decomposes
   \[ H^2_{tr}(S)\cong H^2_{tr}(X_0)\oplus H^2_{tr}(X_1)\ ,\]
   where $X_0,X_1$ are $K3$ surfaces. In \cite{hor}, using arguments very similar to the present note, I proved there exists a similar splitting on the level of Chow groups.
  
  Several open questions remain, which I hope someone will be able to answer (cf. section \ref{sopen}).

 \vskip0.6cm

\begin{convention} In this article, the word {\sl variety\/} will refer to a reduced irreducible scheme of finite type over $\C$. A {\sl subvariety\/} is a (possibly reducible) reduced subscheme which is equidimensional. 

{\bf By default, all Chow groups will be with rational coefficients}: we will denote by $A_j(X)$ the Chow group of $j$--dimensional cycles on $X$ with $\QQ$--coefficients; for $X$ smooth of dimension $n$ the notations $A_j(X)$ and $A^{n-j}(X)$ are used interchangeably. When dealing with Chow groups with integral coefficients, we will make this clear by writing $A_j(X)_\ZZ$.

The notations $A^j_{hom}(X)$, $A^j_{AJ}(X)$ will be used to indicate the subgroups of homologically trivial, resp. Abel--Jacobi trivial cycles.
For a morphism $f\colon X\to Y$, we will write $\Gamma_f\in A_\ast(X\times Y)$ for the graph of $f$.
The contravariant category of Chow motives (i.e., pure motives with respect to rational equivalence as in \cite{Sc}, \cite{MNP}) will be denoted $\MM_{\rm rat}$.

We use $H^j(X)$ 
to indicate singular cohomology $H^j(X,\QQ)$, and $H_j(X)$ to indicate Borel--Moore homology $H_j^{BM}(X,\QQ)$.
\end{convention}

\section{Preliminaries}

%
%
%
%

  \subsection{Relative K\"unneth projectors}
  
  \begin{lemma}\label{relK} Let $\Ss\to B$ be as in notation \ref{fam}. There exist relative correspondences
    \[ \pi_0^\Ss\ ,\ \ \pi_2^\Ss\ ,\ \ \pi_4^\Ss\ \ \ \in A^2(\Ss\times_B \Ss)\ ,\]
    with the property that for each $b\in B$, the restriction
    \[ \pi_i^\Ss\vert_b := \pi_i^\Ss\vert_{S_b\times S_b}\ \ \ \in H^4(S_b\times S_b) \]
    is the $i$th K\"unneth component. Moreover,
     \[ (\pi_2^\Ss\vert_b)_\ast=\ide\colon\ \ \ A^2_{hom}(S_b)\ \to\ A^2_{hom}(S_b)\ .\]
     \end{lemma}
     
    \begin{proof} This is well--known, and holds more generally for any family of surfaces with $H^1(S_b)=0$. Let $H\in A^1(\Ss)$ be a relatively ample divisor, and let $d:=\deg (H^2\vert_{S_b})$. One defines
      \[ \begin{split}   \pi_0^\Ss &:= {1\over d} (p_1)^\ast(H^2)\ ,\\
                               \pi_4^\Ss &:= {1\over d} (p_2)^\ast(H^2)\ ,\\
                               \pi_2^\Ss &:= \Delta_\Ss - \pi_0^\Ss -\pi_4^\Ss\ \ \ \ \in\ A^2(\Ss\times_B \Ss)\ .\\
                               \end{split}\]
                   It is readily checked this does the job.            
        \end{proof}

 \subsection{Transcendental part of the motive}
 
 \begin{theorem}[Kahn--Murre--Pedrini \cite{KMP}]\label{t2} Let $S$ be any smooth projective surface, and let $h(S)\in\MM_{\rm rat}$ denote the Chow motive of $S$.
There exists a self--dual Chow--K\"unneth decomposition $\{\pi_i\}$ of $S$, with the property that there is a further splitting in orthogonal idempotents
  \[ \pi_2= \pi_2^{alg}+\pi_2^{tr}\ \ \hbox{in}\ A^2(S\times S)_{}\ .\]
  The action on cohomology is
  \[  (\pi_2^{alg})_\ast H^\ast(S)= N^1 H^2(S)\ ,\ \ (\pi_2^{tr})_\ast H^\ast(S) = H^2_{tr}(S)\ ,\]
  where the transcendental cohomology $H^2_{tr}(S)\subset H^2(S)$ is defined as the orthogonal complement of $N^1 H^2(S)$ with respect to the intersection pairing. The action on Chow groups is
  \[ (\pi_2^{alg})_\ast A^\ast(S)_{}= N^1 H^2(S)\ ,\ \ (\pi_2^{tr})_\ast A^\ast(S) = A^2_{AJ}(S)_{}\ .\]  
 This gives rise to a well--defined Chow motive
  \[ h^{tr}_2(S):= (S,\pi_2^{tr},0)\ \subset \ h(S)\ \ \in\MM_{\rm rat}\ ,\]
  the so--called {\em transcendental part of the motive of $S$}.
  \end{theorem}

\begin{proof} Let $\{\pi_i\}$ be a Chow--K\"unneth decomposition as in \cite[Proposition 7.2.1]{KMP}. The assertion then follows from \cite[Proposition 7.2.3]{KMP}.
\end{proof}

\section{Triple $K3$ burgers}

\subsection{Definition}

\begin{definition}\label{def} A surface $S$ is called a {\em triple $K3$ burger\/} if the following conditions are satisfied:

\noindent
(0) $S$ is minimal, of general type, with $q=0$ and $p_g=3$;

\noindent
(\rom1) there exist involutions $\sigma_j\colon S\to S$ ($j=0,1,2$) that commute with one another, and such that the quotients 
  \[  \bar{X}_j:= S/<\sigma_j>\ \ \  (j=0,1,2)\] 
  are birational to a $K3$ surface $X_j$;

\noindent
(\rom2) there is an isomorphism
  \[ \bigl( (p_0)^\ast , (p_1)^\ast, (p_2)^\ast\bigr)\colon\ \ \ H^2(\bar{X}_0,\OO)\oplus H^2(\bar{X}_1,\OO)\oplus H^2(\bar{X}_2,\OO)\ \xrightarrow{\cong}\ 
     H^2(S,\OO)\ ,\]
  where $p_j\colon S\to \bar{X}_j$ denotes the quotient morphism;
  
 \noindent
 (\rom3) the involutions $\sigma_j$ respect the canonical divisor:
  \[  (\sigma_j)^\ast K_S =  K_S\ \ \ ,\ j=0,1,2\ .\]  
  
  \end{definition}

\begin{remark}\label{Htr} Let $\Psi_j\in A^2(X_j\times S)$ ($j=0,1,2$) be the correspondence defined by the diagram
  \[ \begin{array}[c]{ccc}
                 &&  S\\
                 &&\downarrow\\
                 X_j&\to&\bar{X}_j\\
                 \end{array}\]
       where $X_j\to\bar{X}_j$ is a resolution of singularities and $X_j$ is a $K3$ surface.          
                 
       Since the $\bar{X}_j$ have only quotient singularities and quotient singularities are rational, condition (\rom2) of definition \ref{def} is equivalent to asking for an isomorphism
        \[      \bigl( (\Psi_0)_\ast , (\Psi_1)_\ast, (\Psi_2)_\ast\bigr)\colon\ \ \ H^2({X}_0,\OO)\oplus H^2({X}_1,\OO)\oplus H^2({X}_2,\OO)\ \xrightarrow{\cong}\ 
     H^2(S,\OO)\ .\]
     
    Also, since $(\Psi_j)_\ast$ is a homomorphism of Hodge structures, condition (\rom2) is equivalent to an isomorphism
    \[        \bigl( (\Psi_0)_\ast , (\Psi_1)_\ast, (\Psi_2)_\ast\bigr)\colon\ \ \ H^2_{tr}({X}_0)\oplus H^2_{tr}({X}_1)\oplus H^2_{tr}({X}_2)\ \xrightarrow{\cong}\ 
     H^2_{tr}(S)\ .\]  
     (Here, by definition $H^2_{tr}()\subset H^2()$ is the orthogonal complement of the N\'eron--Severi group under the cup product pairing.)

Also, since $(p_j)^\ast H^2(\bar{X}_j)$ is contained in the $\sigma_j$--invariant part of $H^2(S)$, condition (\rom2) is equivalent to the condition
  \begin{equation}\label{this}  H^2_{tr}(S)= H^2_{tr}(S)^{+--}\oplus H^2_{tr}(S)^{-+-}\oplus H^2_{tr}(S)^{--+}\ ,\end{equation}
  where $H^2_{tr}(S)^{+--}$ denotes the part of $H^2_{tr}(S)$ where $\sigma_0$ acts as the identity and $\sigma_1, \sigma_2$ act as minus the identity, and the other summands are defined similarly. 
  
  (This uses some Hodge theory. E.g., let us consider $H^2_{tr}(S)^{++-}$. This is a Hodge substructure of $H^2_{tr}(S)$, and so if it is non--trivial, it must have $\gr^0_F$ of dimension $\ge 1$. But then, as it is contained in the image of $H^2_{tr}(X_0)$, it must have $\gr^0_F$ of dimension $=1$. This implies that 
  \[ (\Psi_0)_\ast H^2_{tr}(X_0)=  H^2_{tr}(S)^{++-}\ ,\]
  as both sides are Hodge substructures of $H^2_{tr}(S)$ with $\dim \gr^0_F=1$. But for the same reason, we have
   \[ (\Psi_1)_\ast H^2_{tr}(X_1)=  H^2_{tr}(S)^{++-}\ ,\]  
   and so 
   \[ (\Psi_0)_\ast H^2_{tr}(X_0)=   (\Psi_1)_\ast H^2_{tr}(X_1)\ \ \ \hbox{in}\ H^2_{tr}(S)\ .\]
   But this is absurd, because it contradicts the surjectivity in condition (\rom2). We conclude that $H^2_{tr}(S)^{++-}$ must be zero. Applying the same reasoning to the other eigenspaces, one arrives at the isomorphism (\ref{this}).)
   
   \end{remark}

\begin{remark} Definition \ref{def} is directly inspired by the definition of Todorov surfaces \cite{Tod}, \cite{Kun}, \cite{Mor}, \cite{Rito}. 

One could extend definition \ref{def} to surfaces of any geometric genus: a surface $S$ is called an $m$--tuple $K3$ burger if $p_g(S)=m$ and there exist $m$ involutions $\sigma_1,\ldots,\sigma_m$ such that the quotients $S/<\sigma_j>$ are birational to $K3$ surfaces and their transcendental cohomology generates $H^2_{tr}(S)$ as in condition (\rom2). For $m=1$ (i.e., ``simple $K3$ burgers''), one obtains certain Todorov surfaces. (NB: There is a slight difference with the definition of Todorov surfaces; in the definition of a Todorov surface one merely asks, instead of (\rom3), that the involution $\sigma$ is composed with the bicanonical map).

Surfaces similar to the case $m=2$ of definition \ref{def} (i.e., ``double $K3$ burgers'') have been studied in \cite{Gar}, \cite{PZ}, \cite{hor}.
\end{remark}

\begin{remark} A closely related construction (which also inspired the present note) appears in recent work of Garbagnati \cite[Section 6.1]{Gar}.  
Let $S$ be the minimal model of the surface $U_{10}$ of \cite[Section 6.1]{Gar}. Then $S$ satisfies conditions (0), (\rom1) and (\rom2) of definition \ref{def} (and I am not sure about condition (\rom3)). Also, it follows from \cite[Theorem 3.1]{Gar} that $K_S^2=9$, and so $S$ is not among the cases covered by theorem \ref{main}.

The fact that $K_S^2=9$ means that $S$ is quite far from the Noether line; hence there is (as far as I am aware) not a nice and simple, Horikawa--style description of the canonical model of $S$ as a weighted complete intersection. Due to the lack of such a description, the method of ``spread'' does not seem to apply to $S$, and I do not know how to handle the Chow groups of $S$.
\end{remark}

\begin{remark} Condition (\rom3) in definition \ref{def} is admittedly somewhat ad hoc. The reason I have added condition (\rom3) is that otherwise, I am not able to prove theorem \ref{main}. 

(More precisely: condition (\rom3) ensures that the involutions $\sigma_j$ come from involutions of the ambient space (which will be a weighted projective space); as such, the involutions exist family--wise, which will be crucial to the argument.)
\end{remark}

\begin{remark} Todorov surfaces have been classified: there are $ 11$ irreducible families, each of dimension $12$ \cite{Mor}. Likewise, it is perhaps possible to classify triple $K3$ burgers. The next subsection provides a first step.
\end{remark}

\subsection{Structural results}

\begin{notation}\label{nots} Let $\PP$ be some weighted projective space, with weighted homogeneous coordinates $[x_0:x_1:\cdots:x_n]$. We define involutions $s_j\in\aut(\PP)$, $j=0,\ldots,n$, by
  \[ s_j [x_0:\ldots:x_n]  = [x_0:\cdots:-x_j:\ldots:x_n]\ .\]
  Similarly, for $0\le i<j\le n$ we define involutions $s_{ij}\in\aut(\PP)$ by
   \[ s_{ij} [x_0:\ldots:x_n]  = [x_0:\ldots: -x_i:x_{i+1}:\ldots :  -x_j:x_{j+1}:\ldots:x_n]\ .\]
   Similarly, we define involutions $s_{ijk}$ involving $3$ minus signs.
  \end{notation}

\begin{proposition}\label{2} Let $S$ be a triple $K3$ burger with $K^2=2$. 
%
Then $S$ is  isomorphic to a smooth degree $8$ hypersurface in $\PP(1^3,4)$ invariant under $G=<\sigma_0,\sigma_1,\sigma_2>$, where $\{\sigma_0,\sigma_1,\sigma_2\}$ are one of the following:

\noindent
(\rom1) \[ \{\sigma_0,\sigma_1,\sigma_2\} =\{ s_0,s_1,s_2\}\ .\]

\noindent
(\rom2)  \[ \{\sigma_0,\sigma_1,\sigma_2\} =\{ s_0,s_1,s_{01}\}\ .\]

\noindent
(\rom3)  \[ \{\sigma_0,\sigma_1,\sigma_2\} =\{ s_{01},s_{02},s_{0}\}\ .\]

\noindent
(\rom4)  \[ \{\sigma_0,\sigma_1,\sigma_2\} =\{ s_{01},s_{02},s_{12}\}\ .\]

Conversely, any such surface $S$ is a triple $K3$ burger with $K^2=2$, and the associated $K3$ surfaces are obtained as $\bar{X_j}=S/<\sigma_j>  $, where the $\sigma_j$ are as in (\rom1)--(\rom4).
\end{proposition}                            

\begin{proof} Since $S$ is minimal, of general type, with $K^2=2$ and $p_g=3$, we know that $S$ is isomorphic to a smooth degree $8$ hypersurface in $\PP:=\PP(1^3,4)$
\cite{Hor}. Since the involutions $\sigma_j$ ($j=0,1,2$) preserve the polarization $K_S$, they are induced by involutions of $\PP$. Let $[x_0:x_1:x_2:x_3]$ be weighted homogeneous coordinates for $\PP$. After a projective transformation, we may suppose the involutions are defined by adding a minus sign in front of one or two or three of the $x_i$, i.e. the $\sigma_j$ are of the form $s_i, s_{ij}, s_{012}$, where $i,j\in\{0,1,2\}$.

Griffiths residue calculus (which also exists for weighted projective hypersurfaces, cf. \cite{Dol}, \cite{BC}) shows that $H^{0,2}(S)$ is generated by the image under the residue map of the holomorphic forms with poles 
  \begin{equation}\label{gen}  x_0 \Omega/f\ ,\ \ \ x_1 \Omega/f\ ,\ \ \ x_2 \Omega/f\ .\end{equation}
Here, $f$ is a defining equation for $S$ and $\Omega$ is the standard $3$--form 
 \[ \Omega:=\sum_{i=0}^2 (-1)^i x_i dx_0\wedge\ldots \hat{dx_i}\ldots dx_3 - 4 x_3 dx_0\wedge dx_1\wedge dx_2 \ \]
 \cite[2.1.3]{Dol}, \cite[Example 9.4]{BC}.

The involution $s_{012}$ acts as $-1$ on the form $\Omega$. Hence, the involution $s_{012}$ acts either as $(+1,+1,+1)$ or as $(-1,-1,-1)$ on the three generators (\ref{gen}) (depending on whether $s_{012}$ acts as $+1$ or as $-1$ on $f$).
As such, the quotient $S/<s_{012}>$ can not be a $K3$ surface, and so $s_{012}$ is not among the $\sigma_j$.

Suppose now the $\sigma_j$ are all of type $s_i$. The involution $s_i$ acts on $\Omega$ as $-1$, and on $f$ as $\pm 1$. Considering the action on generators (\ref{gen}), clearly the only possibility is (\rom1).

Suppose next that exactly one of the $\sigma_j$ is of type $s_{ij}$ (and so the others are of type $s_i$). Up to a coordinate change, we may suppose $\sigma_2=s_{01}$. The involution $ s_{01}$ acts on $\Omega$ as $+1$, and on $f$ as $\pm 1$. Since the quotient $S/<s_{01}>$ is $K3$, the action on $f$ has to be the identity, and so $s_{01}$ acts on the generators (\ref{gen}) as $(-1,-1,+1)$. Clearly, the only possibility for $\{\sigma_0,\sigma_1\}$ is now $\{s_0,s_1\}$, and so we are in case (\rom2).

Next, let us suppose that exactly two of the $\sigma_j$ are of type $s_{ij}$, say $\sigma_0=s_{01}$ and $\sigma_1=s_{02}$. 
As per above, the case $s_{ij}(f)=-f$ can be excluded.
We conclude that $\sigma_0$ acts on the generators (\ref{gen}) as $(-1,-1,+1)$, and $\sigma_1$ acts as $(-1,+1,-1)$. The remaining involution $\sigma_2=s_i$ should act as $(+1,-1,-1)$, and so $\sigma_2=\sigma_0$, and we are in case (\rom3).

Finally, if all three $\sigma_j$ are of type $s_{ij}$, they need to be different (for otherwise, there is a generator (\ref{gen}) not preserved by any of the $\sigma_j$). Hence, we are in case (\rom4). 

The converse is clear from the above argument. (Note that the involutions $\sigma_j$ commute because they commute as automorphisms of $\PP$.)
 \end{proof}

\begin{remark} Triple $K3$ burgers as in proposition \ref{2}(\rom1) form a family of moduli dimension $6$. 
Indeed, after a change of variables the equation defining $S$ is of the form
   \[   (x_3)^2= f(x_0,x_1,x_2) \ ,\]
  i.e. $S$ is a double cover of the plane branched along an octic $f$, where $x_0,x_1,x_2$ occur only in even degrees. This family has $6$ moduli. 
     
(The degree $8$ equation
  \[   (x_3)^2= f(x_0,x_1,x_2) \ \]  
  (with $x_0,x_1,x_2$ occurring in even degree)
  depends on $15$ parameters, so smooth hypersurfaces of this type correspond to an open in $\PP^{14}$. The group $PGL(3)$ acts on these hypersurfaces, and so we get $14-8=6$ moduli.)

One element in this family is the weighted Fermat hypersurface
  \[ x_0^8 + x_1^8 +x_2^8 + x_3^2=0\ .\]
  
  The surfaces of proposition \ref{2}(\rom3) and (\rom4) are the same family as that of (\rom1); only the associated $K3$ surfaces are different, so there are different ``burger structures'' on elements of this family.
   \end{remark}

\begin{proposition}\label{3} Let $S$ be a triple $K3$ burger with $K^2=3$ and such that the canonical divisor is base--point free. Then $S$ is  isomorphic to a smooth degree $6$ hypersurface in $\PP(1^3,2)$ invariant under $G=<\sigma_0,\sigma_1,\sigma_2>$, where $\{\sigma_0,\sigma_1,\sigma_2\}$ are one of the following:

\noindent
(\rom1) \[ \{\sigma_0,\sigma_1,\sigma_2\} =\{ s_0,s_1,s_2\}\ .\]

\noindent
(\rom2)  \[ \{\sigma_0,\sigma_1,\sigma_2\} =\{ s_0,s_1,s_{01}\}\ .\]

\noindent
(\rom3)  \[ \{\sigma_0,\sigma_1,\sigma_2\} =\{ s_{01},s_{02},s_{0}\}\ .\]

\noindent
(\rom4)  \[ \{\sigma_0,\sigma_1,\sigma_2\} =\{ s_{01},s_{02},s_{12}\}\ .\]

Conversely, any such surface $S$ is a triple $K3$ burger with $K^2=3$, and the associated $K3$ surfaces are obtained as $\bar{X_j}=S/<\sigma_j>  $, where the $\sigma_j$ are as in (\rom1)--(\rom4).
  \end{proposition}      

\begin{proof} Since $S$ is minimal, of general type, with $K^2=p_g=3$ and base point free canonical divisor, we know that $S$ is isomorphic to a degree $6$ hypersurface in $\PP(1^3,2)$
\cite{Ili}. 

To classify the possible involutions, one proceeds exactly as in the proof of proposition \ref{2}.
                        \end{proof}

\begin{remark} Triple $K3$ burgers with $K^2=3$ and $K_S$ base--point free form a family of dimension $4$. 
(Indeed, under the natural map
  \[ \PP(1^3,2)\ \to\ \PP(2^4)\ ,\]
the hypersurfaces as in proposition \ref{3} correspond to degree $6$ hypersurfaces in $\PP(2^4)$. But under the natural isomorphism
   \[  \PP(2^4)\ \xrightarrow{\cong}\ \PP(1^4)=\PP^3\ ,\]
  the degree $6$ hypersurfaces in $\PP(2^4)$ correspond to degree $3$ hypersurfaces in $\PP^3$, for which there are $4$ moduli.)

We note that there is a subfamily given by triple covers of the plane, and this subfamily has moduli dimension $1$.

(The degree $6$ equation
  \[  (x_3)^3= f(x_0,x_1,x_2) \ \]
  with $x_0, x_1,x_2$ occurring in even degree
  depends on $10$ parameters. We get $10-1-\dim PGL(3)=1$.)
  
  One element in the family (which is also in the subfamily of triple planes) is given by the weighted Fermat hypersurface
   \[  x_0^6+x_1^6+x_2^6+x_3^3=0\ .\]
\end{remark}

\begin{remark} I have not been able to classify triple $K3$ burgers with $K^2=3$ without the assumption that $K_S$ be base point free. When $K_S$ is {\em not\/} base--point free, it is known \cite{Ili} there is exactly one base--point, and the canonical model of $S$ is isomorphic to a bidegree $(3,6)$ complete intersection in 
$\PP(1^3,2,3 )$. However, determining the possible involutions $\sigma_j$ in this case seems to get messy.

Similarly, triple $K3$ burgers with $K^2=4$ and $K_S$ base point free are complete intersections in a weighted projective space \cite{Reid}. I have not been able to classify them.
\end{remark}

%
%
%
%
%

 \subsection{Families}
 
 This section establishes some notation. The two cases in notation \ref{fam} correspond to two cases of propositions \ref{2} and \ref{3}. 
 \begin{notation}\label{fam} Let
   \[ \Ss\ \to\ B \]
   denote one of the following families:
   
  \noindent
  (\rom1) (Case (\rom1) of proposition \ref{2}) The family  
    of all smooth hypersurfaces in $\PP:=\PP(1^3,4)$ of type
   \[ f_b(x_0,x_1,x_2,x_3)=0\ ,\]
   where $f_b$ is weighted homogeneous of degree $8$, and $x_0, x_1, x_2$ occur only in even degree. Let $S_b$ denote the fibre of $\Ss$ over $b\in B$.
   
%
 
%
%
%

 \noindent
 (\rom2) (Case (\rom1) of proposition \ref{3}) The family of all smooth hypersurfaces in $\PP=\PP(1^3,2)$ of type
    \[ f_b(x_0,x_1,x_2,x_3)=0\ ,\]
   where $f_b$ is weighted homogeneous of degree $6$, and $x_0, x_1, x_2$ occur only in even degree. Let $S_b$ denote the fibre of $\Ss$ over $b\in B$.

       \end{notation}
    
    
  \begin{remark} Let $\Ss\to B$ be the family as in notation \ref{fam}(\rom1) (resp. (\rom2)). Then any fibre $S_b$ is a triple $K3$ burger with $K^2=2$ (resp. $K^2=3$). This is immediate from proposition \ref{2} (resp. proposition \ref{3}).
  \end{remark}

 \begin{lemma}\label{smooth} Let $\Ss\to B$ be one of the two families of notation \ref{fam}. The variety $\Ss$ is a smooth quasi--projective variety.
 \end{lemma}
 
 \begin{proof} Let us treat case (\rom1); the other case is similar.
 By construction, there are morphisms
   \[ \begin{array}[c]{ccc} 
      \Ss & \xrightarrow{\pi}& \PP\\
      \downarrow{\scriptstyle \nu} &&\\
      B&&
      \end{array}\]
   Let $\bar{\Ss}\to\bar{B}$ denote the universal family of all (not necessarily smooth) hypersurfaces in $\PP$ of type
    \[ f_b(x_0,x_1,x_2,x_3)=0\ ,\]
    where $f_b$ is weighted homogeneous of degree $8$  and $x_0, x_1,x_2$ only occur in even degrees. Then $\bar{B}$ is a projective space containing $B$ as a Zariski open. 
    
    \begin{lemma}\label{bpf} For any $x\in\PP(1^3,4)$, there exists $b\in\bar{B}$ such that $x\not\in S_b$.
    \end{lemma}
    
   \begin{proof} There is a $(\ZZ/2\ZZ)^3$ cover
     \[ \PP(1^3,4)\ \to\ \PP(2^3,4)\cong \PP(1^3,2)=:\PP^\prime\ .\]
     The surfaces in $\bar{\Ss}\to \bar{B}$ correspond to the complete linear system $\PP H^0(\PP^\prime,\OO_{\PP^\prime}(4))$ which is (ample hence) base point free.
    \end{proof}

    Lemma \ref{bpf} ensures that $\bar{\Ss}$ is a projective bundle over $\PP(1^3,4)$, in particular it is a projective quotient variety. Any surface $S_b$ with $b\in B$ avoids the singular point of $\PP(1^3,4)$, and so $\Ss$ is Zariski open inside a projective bundle over the non--singular locus of $\PP(1^3,4)$.
    It follows that $\Ss$ is smooth. 
       \end{proof}

\section{Trivial Chow groups}

  This intermediate section contains a result asserting the triviality of a certain Chow group. This result (proposition \ref{triv}) will be the most essential ingredient in the proof of our main result (theorem \ref{main} in the next section). The proof of proposition \ref{triv} occupies subsection \ref{ssyes}, and uses a stratification argument borrowed from \cite{tod}. 
  
 \begin{proposition}\label{triv} Let $\Ss\to B$ be a family of surfaces as in notation \ref{fam}. Let $B^0\subset B$ be a Zariski open, and let $\Ss^0\to B^0$
 be the family obtained by restriction. Then
   \[ A^2_{hom}(\Ss^0\times_{B^0} \Ss^0)=0\ .\]
   \end{proposition}

 \subsection{Weak and strong property}
\label{totsection}

  \begin{definition}[Totaro \cite{T}] For any (not necessarily smooth) quasi--projective variety $X$, let $A_i(X,j)$ denote Bloch's higher Chow groups with rational coefficients (these groups are sometimes written $A^{n-i}(X,j)_{\QQ}$ or $CH^{n-i}(X,j)_{\QQ}$, where $n=\dim X$). As explained in \cite[Section 4]{T}, the relation with algebraic $K$--theory ensures there are functorial cycle class maps
    \[  A_i(X,j)_{}\ \to\ \gr^W_{-2i} H_{2i+j}(X)\ ,\]
    compatible with long exact sequences (here $W_\ast$ denotes Deligne's weight filtration on Borel--Moore homology \cite{PS}).
    
    We say that $X$ has the {\em weak property\/} if the cycle class maps induce isomorphisms
    \[   A_i(X)_{}\ \xrightarrow{\cong}\ W_{-2i} H_{2i}(X)\]
    for all $i$.
    
    We say that $X$ has the {\em strong property\/} if $X$ has the weak property, and, in addition, the cycle class maps induce surjections
      \[  A_i(X,1)_{}\ \twoheadrightarrow\ \gr^W_{-2i} H_{2i+1}(X) \]
      for all $i$.
   \end{definition}       
  
 \begin{lemma}[\cite{T}]\label{local0} Let $X$ be a quasi--projective variety, and $Y\subset X$ a closed subvariety with complement $U=X\setminus Y$. If $X$ has the strong property and $Y$ has the weak property, then $U$ has the strong property. 
 \end{lemma} 
 
 \begin{proof} This is \cite[Lemma 6]{T}.
 \end{proof}

\begin{lemma}\label{local} Let $X$ be a quasi--projective variety, and $Y\subset X$ a closed subvariety with complement $U=X\setminus Y$. If $Y$ and $U$ have the strong  property, then so does $X$.
\end{lemma}

\begin{proof} This is the same argument as \cite[Lemma 7]{T}, which is a slightly different statement. As in loc. cit., using the localization property of higher Chow groups \cite{B3}, \cite{Lev}, one finds a commutative diagram with exact rows
  \[ \begin{array}[c]{cccccccc}
                                                       A_{i}(U,1)_{}&\to& A_i(Y)_{}  &\to&      
                                                                                         A_i(X)_{} &\to& A_i(U)_{} & \to 0\\
                                            \downarrow && \downarrow && \downarrow && \downarrow & \\
                                                                       \gr^W_{-2i}H_{2i+1}(U) &\to& \gr^W_{-2i} H_{2i}(Y) &\to& \gr^W_{-2i} H_{2i}(X) &\to& \gr^W_{-2i} H_{2i}(U)& \to 0 \\
 \end{array} \]    
   A diagram chase reveals that under the assumptions of the lemma, the one but last vertical arrow is an isomorphism.
   
   Continuing these long exact sequences to the left, there is a commutative diagram with exact rows
    \[ \begin{array}[c]{cccccccc}


      A_{i}(Y,1)_{}&\to&  A_{i}(X,1)_{}&\to&  A_{i}(U,1)_{}&\to& A_i(Y)_{} &\to\\  
       \downarrow && \downarrow && \downarrow && \ \ \ \downarrow{\cong} & \\
      \gr^W_{-2i} H_{2i+1}(Y)&\to&      \gr^W_{-2i} H_{2i+1}(X)&\to&      \gr^W_{-2i} H_{2i+1}(U) &\to& \gr^W_{-2i} H_{2i}(Y) &\to\\
       \end{array}\]
      
     Chasing some more inside this diagram, one finds that the second vertical arrow is a surjection.       
          \end{proof}

  \begin{corollary}\label{linear} Let $X$ be a quasi--projective variety that admits a stratification such that each stratum is of the form $\A^k\setminus L$, where $L$ is a finite union of linearly embedded affine subspaces. Then $X$ has the strong property.
  \end{corollary}
  
  \begin{proof} Affine space has the strong property (this is the homotopy invariance for higher Chow groups). The subvariety $L$ has the weak property. Doing a diagram chase as in lemma \ref{local} (or directly applying \cite[Lemma 6]{T}), it follows that the variety
  $\A^k\setminus L$  has the strong property. The corollary now follows from lemma \ref{local}.
  \end{proof}
  
 \begin{lemma}\label{projbundle} Let $X$ be a quasi--projective variety with the strong property. Let $Y\to X$ be a projective bundle. Then $Y$ has the strong property.
 \end{lemma}
 
 \begin{proof} This follows from the projective bundle formula for higher Chow groups \cite{B2}.
  \end{proof}

 \subsection{Proof of proposition \ref{triv}}
 \label{ssyes}
 
 \begin{proof}
 
 \noindent
 (\rom1) ($K^2=2$)
 Let us use the shorthand
   \[ \begin{split}  \PP&:=\PP(1^3,4)\ ,\\
                           M&:= \PP\times\PP\ ,\\
                           N&:= \Bigl\{  (f_b,p,p^\prime)\ \in \bar{B}\times \PP\times\PP\ \vert\ f_b(p)=f_b(p^\prime)=0\ \Bigr\}\ \ \ \subset\ \bar{B}\times M\ .
                           \\ \end{split} \]
     The goal is to prove that
     \begin{equation}\label{goal} A^2_{hom}(N)\stackrel{??}{=}0\ .\end{equation}
    
    This implies proposition \ref{triv} for case (\rom1), because (\ref{goal}) implies triviality of $A^2$ of any open in $N$, and $\Ss\times_B \Ss$ is an open in $N$.

 Inside $M$, we have various ``partial diagonals''
  \[      \begin{split} \Delta_M=\Delta_{+++}&:= \Bigl\{ (p,p^\prime)\in\PP\times\PP\ \vert\  p=p^\prime\Bigr\}\ ,\\
                                               \Delta_{+-+}&:= \Bigl\{ (p,p^\prime)\in\PP\times\PP\ \vert\  [p_0:p_1:p_2:p_3] = [p^\prime_0:-p^\prime_1:p^\prime_2:p^\prime_3] \Bigr\}\ ,\\   
  \Delta_{-++}&:= \Bigl\{ (p,p^\prime)\in\PP\times\PP\ \vert\  [p_0:p_1:p_2:p_3] = [-p^\prime_0:p^\prime_1:p^\prime_2:p^\prime_3] \Bigr\}\ ,\\  
    \Delta_{++-}&:= \Bigl\{ (p,p^\prime)\in\PP\times\PP\ \vert\  [p_0:p_1:p_2:p_3] = [p^\prime_0:p^\prime_1:-p^\prime_2:p^\prime_3] \Bigr\}\ ,\\
    \Delta_{+--}&:=  \Bigl\{ (p,p^\prime)\in\PP\times\PP\ \vert\  [p_0:p_1:p_2:p_3] = [p^\prime_0:-p^\prime_1:-p^\prime_2:p^\prime_3] \Bigr\}\ ,\\
       \Delta_{-+-}&:= \Bigl\{ (p,p^\prime)\in\PP\times\PP\ \vert\  [p_0:p_1:p_2:p_3] = [-p^\prime_0:p^\prime_1:-p^\prime_2:p^\prime_3] \Bigr\}\ ,\\
      \Delta_{--+}&:= \Bigl\{ (p,p^\prime)\in\PP\times\PP\ \vert\  [p_0:p_1:p_2:p_3] = [-p^\prime_0:-p^\prime_1:p^\prime_2:p^\prime_3] \Bigr\}\ ,\\
   \Delta_{---}&:= \Bigl\{ (p,p^\prime)\in\PP\times\PP\ \vert\  [p_0:p_1:p_2:p_3] = [p^\prime_0:p^\prime_1:p^\prime_2:-p^\prime_3] \Bigr\}\ ,\\
   \end{split} \]
    
    (Here, we write $p=[p_0:p_1:p_2:p_3]$ and $p^\prime=[p^\prime_0:p^\prime_1:p^\prime_2:p^\prime_3]$. We observe that the various $\Delta_{\pm\mp\pm}$ are just the graphs of the elements of the group $(\ZZ/2\ZZ)^3=<\sigma_0,\sigma_1,\sigma_2>\subset\aut(\PP)$.) 
    
    Let us define the Zariski opens
    \[ \begin{split}   M^0&:= M\setminus (\cup \Delta_{\pm\mp\pm})\ ,\\
                     N^0&:=  N\setminus \pi^{-1} (\cup \Delta_{\pm\mp\pm})\ .\\
                     \end{split}    \]
   
   Corollary \ref{linear} implies that the union  $\cup \Delta_{\pm\mp\pm}$ has the strong property. Since $M=\PP\times\PP$ has the strong property, so does $M^0$ (lemma \ref{local0}).
   The morphism from $N^0$ to $M^0$ has constant dimension (lemma \ref{2cond}), so it is a projective bundle and $N^0$ also has the strong property (lemma \ref{projbundle}).
   
   \begin{lemma}\label{2cond} Let 
     \[ (p,p^\prime)\ \ \ \in \ M\setminus (\cup \Delta_{\pm\mp\pm})\ .\]
     Then $(p,p^\prime)$ imposes $2$ independent conditions on $\bar{B}$, i.e. there exists $b\in\bar{B}$ such that $S_b$ contains $p$ but not $p^\prime$.
     \end{lemma}
     
     \begin{proof} Consider the map
       \[ r \times r \colon\ \ \   M=\PP\times\PP\ \to\ \PP^\prime\times\PP^\prime\ ,\]
       where $\PP^\prime$ is as before $\PP(2,2,2,4)$. The condition $(p,p^\prime)\not\in(\cup \Delta_{\pm\mp\pm})$ implies that $r(p)\not=r(p^\prime)$. Since $\PP^\prime$ is isomorphic to
       $\PP^{\prime\prime}:=\PP(1,1,1,2)$ (and sections of $\OO_{\PP^\prime}(8)$ correspond under this isomorphism to sections of ${\mathcal O}_{\PP^{\prime\prime}}(4)$), lemma \ref{delorme} below shows there exists $S_b$ separating the points $p$ and $p^\prime$.                                                                      
                                                         
    \begin{lemma}\label{delorme} Let $\PP^{\prime\prime}$ be the weighted projective space $\PP(1,1,1,2)$. Then the line bundle ${\mathcal O}_{\PP^{\prime\prime}}(4)$ is very ample.
    \end{lemma}
    
    \begin{proof} The coherent sheaf ${\mathcal O}_{\PP^{\prime\prime}}(4)$ is locally free, because $4$ is a multiple of the weights \cite{Dol}. To see that this line bundle is very ample, we use the following numerical criterion: 
 
 \begin{proposition}[Delorme \cite{Del}]\label{del} Let $P=\PP(q_0, q_1,\ldots,q_n)$ be a weighted projective space. Let $m$ be the least common multiple of the $q_j$. Suppose every monomial
 \[x_0^{b_0} x_1^{b_1}\cdots x_n^{b_n}\]
 of (weighted) degree $km$ ($k\in \NN^\ast$) is divisible by a monomial of (weighted) degree $m$. Then ${\mathcal O}_{P}(m)$ is very ample.
 \end{proposition}
 
(This is the case $E(x)=0$ of \cite[Proposition 2.3(\rom3)]{Del}.)

Using proposition \ref{del}, lemma \ref{delorme} is now easily established.
      \end{proof}   
      \end{proof}

 Let us now finish the proof of proposition \ref{triv} for case (\rom1). Any point 
   \[ (p,p^\prime)\ \ \in\ M^1:= (\cup \Delta_{\pm\mp\pm})\ \ \ \subset M \]
   imposes exactly one condition on $\bar{B}$; indeed $p$ imposes one condition (lemma \ref{bpf}), and since $r(p)=r(p^\prime)$ in $\PP^\prime=\PP(2,2,2,4)$, any $S_b$ containing $p$ also contains $p^\prime$. This means that $N^1$ has the structure of a projective bundle over $M^1$. We have seen above that $M^1$ has the strong property. It follows from lemma \ref{projbundle} that
    \[ N^1:=\pi^{-1}(M^1)\ \ \ \subset N \]
    has the strong property. Lemma \ref{local} now implies that $N$ has the strong property, and so equality (\ref{goal}) is proven. 
    
  \noindent
  (\rom2) ($K^2=3$). Similar to case (\rom1), except that $\PP$ is now $\PP(1^3,2)$ and the degree of the hypersurfaces is $6$. Instead of lemma \ref{delorme}, we now use that $\OO_{\PP^3}(3)$ is very ample.

     \end{proof}

\section{Main}

\begin{theorem}\label{main} Let $S$ be a triple $K3$ burger, and let $X_j (j=0,1,2)$ be the associated $K3$ surfaces. Assume that either

\noindent
(\rom1) $K_S^2=2$, or

\noindent
(\rom2) $K_S^2=3$ and the canonical map is base point free.
 
 Then there is an isomorphism
   \[ (\Psi_0)_\ast+(\Psi_1)_\ast+(\Psi_2)_\ast\colon\ \ \ A^2_{hom}(X_0)\oplus A^2_{hom}(X_1)\oplus A^2_{hom}(X_2)\ \xrightarrow{\cong}\ A^2_{hom}(S)\ .\] 
 \end{theorem}

\begin{proof} First, a reduction step. Let us define eigenspaces
  \[ A^2(S)^{\pm\mp\pm}:=\bigl\{ a\in A^2(S)\ \vert\ (\sigma_{0})^\ast(a)=\pm a\ ,\ (\sigma_{1})^\ast(a)=\mp a\ ,\ (\sigma_{2})^\ast(a)=\pm a\bigr\}\ .\]
  
  We now make the following claim:
  
  \begin{claim}\label{claim} Let $S$ be as in theorem \ref{main}. Any eigenspace with an odd number of minus signs is trivial, i.e.
  \[ A^2(S)^{---}=A^2(S)^{-++}=A^2(S)^{+-+}=A^2(S)^{++-}=0\ .\]
  Moreover,
   \[ A^2_{hom}(S)^{+++}=0\ .\]
  \end{claim}
  
  Before proving the claim, let us verify that the claim suffices to prove the theorem: the claim implies there is a decomposition
  \begin{equation}\label{aneq} A^2_{hom}(S)= A^2_{hom}(S)^{+--}\oplus A^2_{hom}(S)^{-+-}\oplus A^2_{hom}(S)^{--+}\ .\end{equation}
  Also, since necessarily
    \[    (\Psi_0)_\ast A^2(S)\ \ \subset\ A^2(S)^{+\pm\pm}\ ,\]
    the claim implies that
    \[ (\Psi_0)_\ast  A^2_{hom}(S)\ \ \subset\ A^2(S)^{+--}\ .\]  
    What's more, since
    \[ (\Psi_0)_\ast (\Psi_0)^\ast=2 \ide\colon\ \ \ A^2(S)^{+\pm\pm}\ \to\ A^2(S)^{+\pm\pm}\ ,\]
    there is actually equality
    \[  (\Psi_0)_\ast (\Psi_0)^\ast A^2_{hom}(S)= A^2(S)^{+--}\ .\]      
        (And similarly, for reasons of symmetry, 
    \[ \begin{split}   &(\Psi_1)_\ast (\Psi_1)^\ast A^2_{hom}(S)=A^2(S)^{-+-}\ ,\\
                       & (\Psi_2)_\ast (\Psi_2)^\ast A^2_{hom}(S)= A^2(S)^{--+}\  .)\\
                       \end{split}\]
               
               Therefore, the decomposition (\ref{aneq}) is equivalent to the decomposition
         \[ A^2_{hom}(S) =   (\Psi_0)_\ast (\Psi_0)^\ast    A^2_{hom}(S)  \oplus   (\Psi_1)_\ast (\Psi_1)^\ast A^2_{hom}(S)  \oplus (\Psi_2)_\ast (\Psi_2)^\ast A^2_{hom}(S)    \ .\]
         This proves the surjectivity statement of the theorem  
        
        Again using the claim, one deduces that the composition
     \[     \begin{split}  A^2_{hom}(X_0)\oplus A^2_{hom}(X_1)\oplus A^2_{hom}(X_2)\ \xrightarrow{ (\Psi_0)_\ast+(\Psi_1)_\ast+(\Psi_2)_\ast }\ A^2_{hom}(S)&\\     
           \ \xrightarrow{((\Psi_0)^\ast,(\Psi_1)^\ast,(\Psi_2)^\ast)}\  A^2_{hom}(X_0)\oplus A^2_{hom}(X_1)&\oplus A^2_{hom}(X_2)\\
           \end{split} \]
           equals twice the identity. This proves the injectivity statement of the theorem.
           
 It remains to prove the claim. 
 First, let us treat case (\rom2) of propositions \ref{2} and \ref{3}. In this case,
 $\sigma_2=\sigma_0\circ\sigma_1$ (i.e., $G:=<\sigma_0,\sigma_1,\sigma_2>\cong(\ZZ/2\ZZ)^2$), and so the first part of the claim is trivially true. 
 The second part of the claim is also true for these cases: indeed, there is equality 
   \[ A^2_{hom}(S)^{+++}=A^2_{hom}(S/G)\ .\]
  But the surface $S/G$ is a degree $8$ hypersurface in $\PP(1,2,2,4)$ (resp. a degree $6$ hypersurface in $\PP(1,2^3)$), and so $S/G$ is a surface with quotient singularities and ample anticanonical bundle. Such surfaces are rational \cite[Theorem 2.3]{Zh}, and hence $A^2_{hom}(S/G)=0$.
 
 Next, let us consider the cases (\rom1), (\rom3) and (\rom4) of propositions \ref{2} and \ref{3}. In this case, the surfaces $S_b$ are elements of the families of notation \ref{fam}. 
The argument, in a nutshell, is now as follows: the correspondences $\Psi_j$ exist as relative correspondences for the whole family of triple $K3$ burgers. Using the trivial Chow groups result (proposition \ref{triv}), one can upgrade a vanishing in cohomology to a vanishing of Chow groups.

We now proceed to prove claim \ref{claim} for surfaces as in proposition \ref{2}(\rom1), (\rom3) and (\rom4). (The cases of proposition \ref{3}(\rom1), (\rom3) and (\rom4) are mostly the same, modulo some mutatis mutandis which we will indicate below).

\vskip0.3cm
\noindent
\underline{\it Cases (\rom1), (\rom3), (\rom4) of proposition \ref{2}:} 
Let 
  \[ \Ss\ \to\ B \]
  denote the universal family of surfaces as in notation \ref{fam}(\rom1). Let $\{\sigma_0,\sigma_1,\sigma_1\}$ be either $\{s_0,s_1,s_2\}$ or $\{s_{01},s_{02},s_0\}$, and let
    \[ \XX_j := \Ss/\sigma_j\ \ \ (j=0,1,2) \]
  denote the universal families of associated $K3$ surfaces as in notation \ref{fam}.
  For any $b\in B$, we will write $S_b$ for the fibre of $\Ss$ over $b$, and $X_{0b}$ (resp. $X_{1b}$ resp. $X_{2b}$) for the fibre of $\XX_{0}$ (resp. $\XX_1$ resp. $\XX_2$) over $b$. Likewise, we will write $\sigma_{0b}$, $\sigma_{1b}$, $\sigma_{2b}$ for the restriction of $\sigma_0$ (resp. $\sigma_1$ resp. $\sigma_2$) to $S_b$.
  For a relative correspondence $\Gamma\in A^\ast(\Ss\times_B \Ss)$, we will use the shorthand
  \[ \Gamma\vert_b:=\Gamma\vert_{S_b\times S_b}\ \ \in A^\ast(S_b\times S_b) \]
  for the restriction (i.e., the image of $\Gamma$ under the Gysin homomorphism induced by the inclusion $b\hookrightarrow B$).

 By definition (cf. remark \ref{Htr}), we know that there is a fibrewise isomorphism
   \begin{equation}\label{fibrewise} \begin{split}  H^2_{tr}(S_b)&\cong H^2_{tr}(S_b)^{+--}\oplus H^2_{tr}(S_b)^{-+-}\oplus H^2_{tr}(S_b)^{--+}\\
       &\cong H^2_{tr}(X_{0b})\oplus H^2_{tr}(X_{1b})\oplus H^2_{tr}(X_{2b})\ \ \ \forall b\in B\ .\\
       \end{split}\end{equation} 

That is, there are no eigenspaces with an odd number of minus signs:

 \begin{equation}\label{noodd}
     H^2_{tr}(S_b)^{---} = H^2_{tr}(S_b)^{-++}=H^2_{tr}(S_b)^{+-+}=H^2_{tr}(S_b)^{++-}=0\ \ \ \forall b\in B\ .
     \end{equation}
     
  Also, there is no eigenspace without minus signs:
  
  \begin{equation}\label{no0} H^2_{tr}(S_b)^{+++}=0\ \ \ \forall b\in B\ . \end{equation}   
  
  Let us define a relative correspondence
  \[ \Gamma^{---}:= {1\over 8} ( \Delta_{\Ss}-\Gamma_{\sigma_0}) \circ(\Delta_{\Ss}-\Gamma_{\sigma_1})\circ(\Delta_\Ss-\Gamma_{\sigma_2})\circ \pi_2^{\Ss}\ \ \ \in A^2(\Ss\times_B \Ss) \ .\]

        (For details on the formalism of relative correspondences and their composition, cf. \cite[Chapter 8]{MNP} whose conventions are met with in our set--up.)
        
       We observe that for any $b\in B$, the restriction  
       \[ \Gamma^{---}\vert_b\ \ \ \in A^2(S_b\times S_b) \]
       is a projector on $H^2(S_b)^{---}$.

In terms of correspondences, the vanishing $H^2_{tr}(S_b)^{---}=0$ in (\ref{noodd}) is equivalent to the statement that
 \[   \bigl(  \Gamma^{---}\vert_b  \bigr) \circ \pi^{S_b}_{2,tr}  =0 \ \ \ \hbox{in}\ H^4(S_b\times S_b) \ \ \ \forall b\in B \ .  \]
 (Here, $\pi^{S_b}_{2,tr}$ is a projector defining the transcendental part of the motive as in theorem \ref{t2}.)
 This is in turn equivalent to the statement that for any $b\in B$, there exists a divisor $D_b\subset S_b$, and a cycle $\gamma_b$ supported on $D_b\times D_b\subset S_b\times S_b$, such that
  \[ \bigl( \Gamma^{---}\vert_b \bigr)   \circ \pi^{S_b}_2    = \gamma_b\ \ \ \hbox{in}\ H^4(S_b\times S_b)\ .\]
  Using a Baire category argument as in \cite[Proposition 3.7]{V0} or \cite[Lemma 1.4]{V1}, these data can be ``spread out'' over the base $B$, i.e. one can find a divisor $\DD\subset\Ss$ and a cycle $\gamma$ supported on $\DD\times_B \DD\subset \Ss\times_B \Ss$ such that
   \[ \bigl( \Gamma^{---} \circ \pi_2^{\Ss}\bigr)\vert_b  = \gamma\vert_b\ \ \ \hbox{in}\ H^4(S_b\times S_b)\ \ \ \forall b\in B\ .\]
  In other words, the relative correspondence
   \[ \Gamma:=  \Gamma^{---}\circ \pi_2^{\Ss}  -\gamma\ \ \ \in A^2(\Ss\times_B \Ss) \]
   is fibrewise homologically trivial:
   \[ \Gamma\vert_b\ \ \ \in A^2_{hom}(S_b\times S_b)\ \ \ \forall b\in B\ .\]
   
   The next step is to make $\Gamma$ globally homologically trivial. Employing a Leray spectral sequence argument as in \cite[Lemmas 3.11 and 3.12]{V0}, this can be done by adding a cycle coming from the ambient space $\PP$. More precisely, the argument of \cite[Lemmas 3.11 and 3.12]{V0} proves the following: up to shrinking the base (i.e., after replacing $B$ by a dense Zariski open $B^\prime\subset B$, and writing $B:=B^\prime$ for simplicity), there exists
    $\delta\in A^2(\PP\times\PP)$ such that 
   \[  \Gamma + (\delta\times B)\vert_{\Ss\times_B \Ss}\ \ \ \in A^2_{hom}(\Ss\times_B \Ss)\ .\]
   
   In view of the fact that $A^2_{hom}(\Ss\times_B \Ss)=0$ (proposition \ref{triv}), it follows that   
    \[  \Gamma + (\delta\times B)\vert_{\Ss\times_B \Ss}=0\ \ \ \hbox{in}\ A^2_{}(\Ss\times_B \Ss)\ .\]
   We know that for any $b\in B$, the restriction $\delta\vert_b$ acts trivially on $A^2_{hom}(S_b)$ (the action factors over $A^\ast_{hom}(\PP)=0$). The above thus implies in particular that
   \[ (\Gamma\vert_b)_\ast=0\colon\ \ \ A^2_{hom}(S_b)\ \to\ A^2_{hom}(S_b)\ \ \ \forall b\in B\ .\]
  By definition of $\Gamma$, this means that
   \[ \bigl(  \Gamma^{---}\vert_b  -\gamma\vert_b\bigr){}_\ast =0\colon\ \ \ A^2_{hom}(S_b)\ \to\ A^2_{hom}(S_b)\ \ \ \forall b\in B\ .\]
       Since for $b\in B$ general, the restriction $\gamma\vert_b$ will still be supported on (divisor)$\times$(divisor), we know that
   \[ (\gamma\vert_b)_\ast=0\colon\ \ \ A^2_{hom}(S_b)\ \to\ A^2_{hom}(S_b)\ \ \ \hbox{for\ general\ }b\in B\ .\]
   Thus, the above simplifies to
   \[ \bigl(  (\Gamma^{---}\circ \pi_2^{\Ss})\vert_b  \bigr){}_\ast =0\colon\ \ \ A^2_{hom}(S_b)\ \to\ A^2_{hom}(S_b)\ \ \  \hbox{for\ general\ }b\in B\ .\]
   Using a Baire category argument as in \cite[Lemma 3.1]{LFu}, this can be extended to {\em all\/} elements of the base $B$: we actually have
    \[ \bigl(  (\Gamma^{---}\circ \pi_2^{\Ss}\vert_b  \bigr){}_\ast =0\colon\ \ \ A^2_{hom}(S_b)\ \to\ A^2_{hom}(S_b)\ \ \ \forall b\in B\ ,\]  
    where $B$ is now once more (as in the beginning of the proof) the parameter space parametrizing {all\/} triple $K3$ burgers as in notation \ref{fam}. 
      
 By construction $\Gamma^{---}\vert_b$ acts on $A^2(S_b)^{---}$ as the identity, and
  \[  (\Gamma^{---}\circ \pi_2^\Ss)\vert_b = \Gamma^{---}\vert_b\circ \pi_2^{S_b} \]
  acts on $A^2_{hom}(S_b)^{---}$ as the identity.
  The above thus implies the vanishing
  \[ A^2_{hom}(S_b)^{---} = 0\ \ \ \forall b\in B\ ,\]
  which proves the first part of the claim. The other parts of the claim are proven similarly, by choosing a different correspondence:
  e.g., for the second vanishing statement one considers the relative correspondence
   \[  \Gamma^{-++}:={1\over 8} ( \Delta_{\Ss}-\Gamma_{\sigma_0}) \circ(\Delta_{\Ss}+\Gamma_{\sigma_1})\circ(\Delta_{\Ss}+\Gamma_{\sigma_2})\circ \pi_2^{\Ss}\ \ \ \in A^2(\Ss\times_B \Ss)  \ .\]
  

\vskip0.3cm
\noindent
\underline{\it Cases (\rom1), (\rom3), (\rom4) of proposition \ref{3}:} 
The claim is proven by the same argument as in case (\rom1), applied to the family $\Ss\to B$ as specified in notation \ref{fam}. The weighted projective space $\PP$ now has different weights, and the defining equation has a different degree. The trivial Chow groups statement (proposition \ref{triv}) still holds for this family.
\end{proof}

\section{Corollaries}

\subsection{The canonical $0$--cycle}
\label{scan}

In this subsection, we work with integral Chow groups $A^i()_\ZZ$, instead of Chow groups with rational coefficients. Let $S$ be a triple $K3$ burger as in theorem \ref{main}.
Thanks to Rojtman's theorem \cite{Roi}, theorem \ref{main} implies that
  \[ A^2(S)_\ZZ^{+++}\cong \ZZ\ .\]

\begin{definition}\label{can} Let $S$ be a triple $K3$ burger as in theorem \ref{main}. The canonical $0$--cycle $\oo_S$ is defined as the unique degree $1$ cycle such that
  \[  A^2(S)_\ZZ^{+++}= \ZZ[\oo_S]\ \]
  (where $A^2(S)_\ZZ^{+++}$ denotes as before the subspace where $\sigma_j$ acts as the identity for $j=0,1,2$).
  
  Equivalently, $\oo_S$ is the unique degree $1$ cycle $z$ satisfying
    \[   (\Psi_j)^\ast(z) =\oo_{X_j}\ \ \ \hbox{in}\ A^2(X_j)_\ZZ\ \ \ (j=0,1,2)\ ,\]
    where $X_j$ are the associated $K3$ surfaces and the correspondences $\Psi_j\in A^2(X_j\times S)_\ZZ$ are as above.
    
   Equivalently, $\oo_S$ is the unique degree $1$ cycle $z$ satisfying
    \[   (\Psi_j)_\ast  (\Psi_j)^\ast(z) =2 z\ \ \ \hbox{in}\ A^2(S)_\ZZ\ \ \ (j=0,1,2)\ .\]
  \end{definition}
  
  The equivalences in definition \ref{can} are valid because of the following lemma:
  
  \begin{lemma}\label{lemma} Let $S$ be a triple $K3$ burger as in theorem \ref{main}.
  Then
      \[  (\Psi_0)_\ast(\oo_{X_0})= (\Psi_1)_\ast(\oo_{X_1}) = (\Psi_2)_\ast (\oo_{X_2})\ \ \ \in\ A^2(S)_{\ZZ}\ .\]
     \end{lemma}
     
     \begin{proof} The point is that there is a commutative diagram of surfaces
     \[ \begin{array}[c]{ccc}
             & S & \\
          \ \  \ \   \swarrow{\scriptstyle p_0} &\ \  \downarrow{\scriptstyle p_1} &\ \ \ \  \searrow {\scriptstyle p_2}\\
          &&\\
           \bar{X}_0\ \  \ \ \ \ & \bar{X}_1 &\ \ \ \ \ \  \bar{X}_2  \\
           &&\\
         \ \ \ \   \searrow {\scriptstyle r_0} &\ \  \downarrow {\scriptstyle r_1} &\ \ \ \  \swarrow{\scriptstyle r_2}\\
            & W & \\
            \end{array}\]
       where all arrows are degree $2$ morphisms, and $A^2(W)_{\ZZ}=\ZZ$.    
      (In case (\rom1) of theorem \ref{main}, the surface $W$ is defined as the degree $8$ hypersurface in $\PP(2^3,4)$ defined by the equation $f(t_0,t_1,t_2,x_3)=0$, where $f(x_0^2,x_1^2,x_2^2,x_3)=0$ is a defining equation for $S$. For cases (\rom2) and (\rom3), the construction is similar.)   
      
     Let us pick two divisors $D, D^\prime$ on $W$, and set
      \[ w:= D\cdot D^\prime\ \ \ \in A^2(W)\ .\]
      The pullbacks to the various $\bar{X}_j$ are intersections of divisors, and so
      \[ (r_j)^\ast (w) = d \, \oo_{\bar{X}_j}\ \ \ \hbox{in}\ A^2(\bar{X}_j)\ \ \ (j=0,1,2)\ .\]
      (Here, $d=\deg (D\cdot D^\prime)$, and we define $\oo_{\bar{X}_j}$ to be $(q_j)_\ast(\oo_{X_j}$.)
      This implies that
      \[ (\Psi_j)_\ast (d\,\oo_{X_j}) = (p_j)^\ast (d\,\oo_{\bar{X}_j})= (r_j\circ p_j)^\ast (w)\ \ \ \hbox{in}\ A^2(S)_\ZZ\ \ \ (j=0,1,2)\ ,\]
      and so 
        \[  d(\Psi_0)_\ast(\oo_{X_0})= d(\Psi_1)_\ast(\oo_{X_1}) = d(\Psi_2)_\ast (\oo_{X_2})\ \ \ \in\ A^2(S)_{\ZZ}\ .\]
        Using Rojtman's theorem \cite{Roi}, this proves the lemma.
      \end{proof}

  We now recall the definition of the ``effective orbit under rational equivalence'' of a $0$--cycle:
  
  \begin{definition}[Voisin \cite{V21}] Let $S$ be any surface. Given a cycle $z\in A^2(S)_\ZZ$ of degree $k\ge 0$, we define the ``effective orbit'' $O_z$ as
    \[ O_z:= \bigcup_{z^\prime\in X^{(k)}, z^\prime\thicksim_{\rm rat} z}  \hbox{supp} (z^\prime) \ \ \ \subset X^{(k)}\ .\]
  (Here, the union is taken over all $k$--tuples of points $z^\prime$ such that the $0$--cycle associated to $z^\prime$ is rationally equivalent to the $0$--cycle $z$ in $X$.)
  
  One defines
    \[ \dim O_z := \sup_{V\subset O_z} \dim V\ ,\]
    where the supremum runs over all irreducible components $V\subset O_z$
   (we note that $O_z$ is known to be a countable union of closed subvarieties, so this is well--defined).
    \end{definition}
  
  Inspired by \cite{V21}, one can give a nice characterization of the canonical $0$--cycle $\oo_S$:
  
  \begin{proposition}\label{orbit} Let $S$ be a triple $K3$ burger as in theorem \ref{main}. Let $k>0$ be an integer. Then $k\oo_S$ is the unique degree $k$ $0$--cycle $z\in A^2(S)_\ZZ$
  satisfying $\dim O_{z}\ge k$.
  \end{proposition}
  
  \begin{proof} We actually prove a somewhat more general statement, which is based on Voisin's result \cite[Theorem 1.4]{V21}. This result of Voisin's gives an alternative description of O'Grady's filtration $S^k_d()$ on the Chow group of $0$--cycles of a $K3$ surface, in terms of effective orbits. We recall that for any $K3$ surface $X$, O'Grady's filtration \cite{OG} is defined as
    \begin{equation}\label{filt} S^k_d(X):=  \bigl\{    z\in A^2(X)_\ZZ\ \vert\ z= z^\prime + (k-d)\oo_X      \bigr\}\ ,\end{equation}
    where $z^\prime$ is effective of degree $d$ and $\oo_X$ is the canonical $0$--cycle.
    
   Voisin gives an interesting alternative description of the O'Grady filtration: for any $k>d\ge 0$, she proves \cite[Theorem 1.4]{V21} that
   \begin{equation}\label{alt} S^k_d(X)=  \bigl\{ z\in A^2(X)_\ZZ\ \vert\ O_z\subset X^{(k)}\not=\emptyset\ \hbox{and}\ \dim O_z\ge k-d\bigr\}\ .
       \end{equation}
       
     Let us now consider a triple $K3$ burger $S$ as in theorem \ref{main}. The canonical $0$--cycle $\oo_S$ exists, and so definition (\ref{filt}) makes sense for $S$.  
     
 \noindent
 {\it Step 1 (Unicity):}
       Let $z\in A^2(S)_\ZZ$ of degree $k$, and let us assume that the orbit $O_z\subset S^{(k)}$ is non--empty of dimension $\ge k-d$, for some $k>d\ge 0$. 
    According to (the proof of) theorem \ref{main}, we can write $z$ uniquely as
    \[ z=k\oo_S +z_0 +z_1 +z_2\ \ \ \hbox{in}\ A^2(S)_\ZZ\ ,\]
    where $z_0\in A^2_{hom}(S)_\ZZ^{+--}$ and $z_1, z_2$ are in $A^2_{hom}(S)_\ZZ^{-+-}$ resp. in $A^2_{hom}(S)_\ZZ^{--+}$.
     
     The assumption on $O_z$ implies that the cycles
     \[   (\Psi_j)^\ast (z) =  k\oo_{X_j} + (\Psi_j)^\ast(z_j)  \in A^2(X_j)_\ZZ\ \ \ (j=0,1,2) \]
     also have orbits $O_{z_j}$ of dimension $\ge k-d$. Therefore, Voisin's result (\ref{alt}) implies that
     \[ (\Psi_j)^\ast (z)\ \ \ \in S^k_d(X_j)\ \ \ (j=0,1,2)\ ,\]
     i.e. one can write
     \[    (\Psi_j)^\ast(z)=  k\oo_{X_j}+  (\Psi_j)^\ast(z_j)  = z^\prime_j + (k-d)\oo_{X_j}\ \ \ \hbox{in}\ A^2(X_j)_\ZZ\ \ \ (j=0,1,2)\ ,\]
     where $z^\prime_j$ is effective of degree $d$. 
     It follows that
     \[  (\Psi_j)^\ast(z_j) = z^\prime_j -d\, \oo_{X_j}\ \ \ \hbox{in}\ A^2(X_j)_\ZZ\ \ \ (j=0,1,2)\ .\]
       
  Using the proof of theorem \ref{main}, we find that
   \[ \begin{split} 2z &= 2k\,\oo_S + 2z_0 +2 z_1 +2z_2\\
           &=    2k\,\oo_S+    (\Psi_0)_\ast(\Psi_0)^\ast(z_0) + (\Psi_1)_\ast(\Psi_1)^\ast(z_1) +   (\Psi_2)_\ast(\Psi_2)^\ast(z_2)  \\
                               &=2k\,\oo_S (\Psi_0)_\ast( z^\prime_0 -d\,\oo_{X_0}           ) + (\Psi_1)_\ast(z^\prime_1-d\,\oo_{X_1}) + (\Psi_2)_\ast(z^\prime_2-d\,\oo_{X_2}) \\
                               &= 2(k-3d)\oo_S + b_0+b_1 +b_2\ \ \ \ \ \ \hbox{in}\ A^2(S)_\ZZ\ ,\\
                           \end{split}\]
                        where $b_0+b_1+b_2$ is effective of degree $6d$. That is, we have
                    \[ 2z\ \ \ \in\ S^{2k}_{6d}(S)\ .\]
 
 In particular, taking $d=0$ we obtain the following implication: if $z$ is a degree $k$ cycle with orbit $O_z$ of dimension $\ge k$, then 
   \[ 2z=2k\,\oo_S\ \ \  \hbox{in}\ A^2(S)_\ZZ\ .\] 
   As $A^2_{hom}(S)_\ZZ$ is torsion free, it follows that
   \[ z=k\,\oo_S\ \ \ \hbox{in}\ A^2(S)_\ZZ\ .\]

  
   \noindent
   {\it Step 2 (Existence):} We now prove that the cycle $z=k\,\oo_S$ has orbit of dimension $\ge k$. This is 
   the easier direction. Take $j\in\{0,1,2\}$, and let $\bar{C}\subset\bar{X}_j$ be any rational curve. 
  Using lemma \ref{lemma}, one finds that the curve $C:=(p_j)^{-1}(\bar{C})\subset S$ is a constant cycle curve, and any point $p\in C$ is such that $(\Psi_j)^\ast (p)=\oo_{X_j}$ and so $p$ represents $\oo_S$. This proves the statement for $k=1$. For $k>1$, one notes that $C^{(k)}\subset S^{(k)}$ is contained in the orbit of $k\,\oo_S$.
                \end{proof}

 Let $Z$ be any smooth projective variety (say of dimension $n$), and let $z\in A^n_{hom}(Z)$ be a degree $0$ $0$--cycle.
 It is known that $z$ is {\em smash--nilpotent\/}, meaning that 
    \[ \begin{array}[c]{ccc}  z^{\times(N)}:= &\undermat{(N\ \hbox{ times})}{z\times\cdots\times z}&=0\ \ \hbox{in}\  A^{Nn}(Z^{n})_{}\ 
  \end{array}\]
   \vskip0.6cm  
   \noindent
  for $N>>0$ \cite{Voe}, \cite{V9}.
  In the special case of the varieties under consideration in this note, one can give a precise estimate for the smash--nilpotence index $N$:

 \begin{proposition}\label{smash} Let $S$ be a triple $K3$ burger as in theorem \ref{main}. Let $z\in A^2_{hom}(S)$ be a $0$--cycle of the form
   \[ z= z^\prime - d\oo_S\ \ \ \in A^2_{hom}(S) \ ,\]
   where $z^\prime$ is an effective cycle of degree $d$. 
   Then
    \[ \begin{array}[c]{ccc}  z^{\times(3d+1)}:= &\undermat{((3d+1)\ \hbox{ times})}{z\times\cdots\times z}&=0\ \ \hbox{in}\  A^{6d+2}(S^{3d+1})_{}\ .
  \end{array}\]
  \end{proposition}
  \vskip0.6cm  
   
  \begin{proof} The assumption means that $z$ is in the subgroup $S^0_d(S)$ of the O'Grady filtration mentioned in the proof of proposition \ref{orbit} above.
This implies that
  \[ (\Psi_j)^\ast(z)\ \ \ \in S^0_d(X_j)\ ,\ \ \ j=0,1,2\ .\]  
  
  For any positive integer $r$, theorem \ref{main2} gives an isomorphism of Chow motives
   \[ t(S^r)\cong \bigoplus_{r_0+r_1+r_2=r} t((X_0)^{r_0})\otimes t((X_1)^{r_1})\otimes t((X_2)^{r_2})\ \ \ \hbox{in}\ \MM_{\rm rat}\ \]
   (induced by the $\Psi_j$),
   and so there is an isomorphism of Chow groups
   \[\label{iso}  \begin{split}  &\sum_{r_0+r_1+r_2=r} \bigl( ((\Psi_0)^{r_0})^\ast,  ((\Psi_1)^{r_1})^\ast,((\Psi_2)^{r_2})^\ast  \bigr)\colon \\
        &\ \ \ \ \ \ \  A^{2r}(t(S)^{\otimes r})\ \xrightarrow{\cong}\ \bigoplus_{r_0+r_1+r_2=r}  A^{2r_0}(t(X_0)^{\otimes r_0})\otimes    A^{2r_1}(t(X_1)^{\otimes r_1})\otimes  A^{2r_2}(t(X_2)^{\otimes r_2})\ .\\
        \end{split}\]
  
  In particular, this implies that there is an injection
   \begin{equation}\label{inj}        \begin{split}  &\sum_{r_0+r_1+r_2=r} \bigl( ((\Psi_0)^{r_0})^\ast,  ((\Psi_1)^{r_1})^\ast,((\Psi_2)^{r_2})^\ast  \bigr)\colon \\
        &\ \ \ \ \ \ \  A^{2r}(t(S)^{\otimes r})\ \hookrightarrow\ \bigoplus_{r_0+r_1+r_2=r}  A^{2r_0}((X_0)^{ r_0})\otimes    A^{2r_1}((X_1)^{ r_1})\otimes  A^{2r_2}((X_2)^{r_2})\ .\\
        \end{split}\end{equation}

   Consider now the element $z^{\times r}$ for $r\ge 3d+1$. Since $z\in A^2_{hom}(S)=A^2(t(S))$, we have
     \[ z^{\times r}\ \ \ \in\ A^{2r}(t(S)^{\otimes r})\ .\]      
            The image of $z^{\times r}$ in the right--hand side of the injection (\ref{inj}) is a sum of $0$--cycles on the various products $(X_0)^{r_0}\times (X_1)^{r_1}\times (X_2)^{r_2}$. In each summand, one of the integers $r_0, r_1, r_2$ must be $\ge d+1$. The proposition now follows from the following lemma:
   
   \begin{lemma}[O'Grady \cite{OG}] Let $X$ be a $K3$ surface, and let $z\in S^0_d(X)$. Then
     \[ z^{\times (d+1)}=0\ \ \ \hbox{in}\ A^{2d+2}(X^{d+1})\ .\]
     \end{lemma}
     
   \begin{proof} This is established in \cite[(5.0.1)]{OG}. The reason is that $z$ can be represented by a degree $0$ $0$--cycle $w$ on a curve $C\subset X$ of geometric genus $d$. This proves the lemma, for it is known since \cite{V9} that $w^{\times (d+1)}=0$ in $A^{d+1}(C^{d+1})$. 
    \end{proof}  
      
     \end{proof}

\subsection{The canonical $0$--cycle, bis}

\begin{definition} Let $S$ be a triple $K3$ burger, and let $X_j$ ($j=0,1,2$) be the associated $K3$ surfaces. By definition, the subgroup of $K3$--type divisors
$A^1_{K3}(S)_\ZZ\subset A^1(S)_\ZZ$ is defined as
  \[  \begin{split} A^1_{K3}(S)_\ZZ:=  \bigl( (\Psi_0)_\ast A^1(X_0)_\ZZ\cap (\Psi_1)_\ast A^1(X_1)_\ZZ\bigr) +  \bigl( (\Psi_0)_\ast A^1(X_0)_\ZZ\cap (\Psi_2)_\ast A^1(X_2)_\ZZ\bigr)&\\
   +
                              \bigl( (\Psi_1)_\ast A^1(X_1)_\ZZ\cap (\Psi_2)_\ast A^1(X_2)_\ZZ\bigr) \ .&\\
                              \end{split}\]
    That is, 
    \[ A^1_{K3}(S)_\ZZ= A^1(S)_\ZZ^{+++}\oplus A^1(S)_\ZZ^{++-}\oplus A^1(S)_\ZZ^{+-+}\oplus A^1(S)_\ZZ^{-++}\ .\]
    \end{definition}

\begin{proposition}\label{int} Let $S$ be a triple $K3$ burger as in theorem \ref{main}. Let $D, D^\prime\in A^1_{K3}(S)_\ZZ$. Then
  \[ D\cdot D^\prime = \deg(D\cdot D^\prime) \, \oo_S\ \ \ \hbox{in}\ A^2(S)_\ZZ\ .\]
  \end{proposition}
  
  \begin{proof} Since $A^2_{hom}(S)_\ZZ$ is torsion free \cite{Roi}, it suffices to prove the statement for Chow groups with $\QQ$--coefficients.
  We have seen that
    \[ A^1_{K3}(S)= A^1(S)^{+++}\oplus A^1(S)^{++-}\oplus A^1(S)^{+-+}\oplus A^1(S)^{-++}\ .\]
    Assuming that $D$ and $D^\prime$ are in the same summand of this decomposition, we have 
    \[ D\cdot D^\prime\ \ \ \in A^2(S)^{+++}=\QQ[\oo_S]\ ,\]
    and we are done.
    
    Next, let us assume $D$ is in the first summand and $D^\prime$ is in another summand (say the second). Then
    \[ D\cdot D^\prime\ \ \ \in A^2(S)^{++-}\ .\]
    But $A^2(S)^{++-}=0$ (proof of theorem \ref{main}), and so $D\cdot D^\prime=0$.
    
    Finally, let us assume $D$ and $D^\prime$ are in two different summands and neither is in the first summand (say $D\in A^1(S)^{+-+}$ and $D^\prime\in A^1(S)^{-++}$). Then
     \[ D\cdot D^\prime\ \ \ \in A^2(S)^{--+}\ .\]
     We have seen (proof of theorem \ref{main}) that $A^2(S)^{--+}$ is mapped isomorphically (under $(\Psi_2)^\ast$) to $A^2_{hom}(X_2)$, and so to prove that $D\cdot D^\prime=0$, it suffices to prove that
      \[ (\Psi_2)^\ast (D\cdot D^\prime)\stackrel{??}{=}0\ \ \ \hbox{in}\ A^2_{hom}(X_2)\ .\]
      
   To this end, recall that (by construction) $(\Psi_2)^\ast = (q_2)^\ast (p_2)_\ast$ (where $p_2\colon S\to\bar{X}_2$ is projection to the $K3$ surface with double points, and $q_2\colon X_2\to \bar{X}_2$ is a resolution of singularities). Hence,
    \[ \begin{split}   (\Psi_2)^\ast (D\cdot D^\prime) &= (q_2)^\ast  (p_2)_\ast (D\cdot D^\prime)\\
                                                                               &=(q_2)^\ast \bigl(\bar{F}\cdot (p_2)_\ast(D^\prime)\bigr)\\
                                                                               &= (q_2)^\ast(\bar{F})\cdot (q_2)^\ast (p_2)_\ast(D^\prime)\\
                                                                               &=0\ \ \ \ \ \ \hbox{in}\ A^2_{hom}(X_2)\ .\\
                                                                               \end{split}\]
                                                          Here, $\bar{F}\in A^1(\bar{X}_2)$ is a divisor such that $D=(p_2)^\ast(\bar{F})$. The last line follows from the celebrated Beauville--Voisin result that 
                      \[ \bigl(A^1(X_2)\cdot A^1(X_2)\bigr) \cap A^2_{hom}(X_2)=0 \] 
                      for any $K3$ surface $X_2$ \cite{BV}.                    
                                                                            \end{proof}

\begin{remark} The behaviour displayed in proposition \ref{int} is remarkable, because the dimension of $A^1_{K3}(S)$ tends to be large. For example, let $S$ be a triple $K3$ burger with $K^2=2$. Then $A^1(S)^{+++}$ coincides with $A^1(T)$, where
  \[ T:= S/<\sigma_0,\sigma_1,\sigma_2>\ .\]
  The surface $T$ can be identified with a degree $4$ hypersurface in $\PP(1^3,2)$. Hence, $T$ is isomorphic to the double cover of $\PP^2$ branched along a quartic curve. In case the quartic curve is smooth, one has $\dim A^1(T)=\dim H^2(T)=8$ \cite{Ste}, and so
  \[ \dim A^1_{K3}(S)\ge \dim A^1(S)^{+++}=8\ .\]
\end{remark}

\subsection{Bloch conjecture}

 \begin{corollary}\label{cor1.5} Let $S$ be a triple $K3$ burger as in theorem \ref{main}, and let $\sigma_0,\sigma_1, \sigma_2$ be the three covering involutions. Let $f\in\aut(S)$ be a finite--order automorphism that commutes with the $\sigma_j$, and such that
   \[ f^\ast=\ide\colon\ \ \ H^{2,0}(S)\ \to\ H^{2,0}(S)\ .\]
   Then also
    \[ f^\ast=\ide\colon\ \ \ A^{2}_{}(S)\ \to\ A^{2}(S)\ .\] 
   \end{corollary}

   \begin{proof} Since $f$ commutes with the $\sigma_j$, $f$ induces finite--order automorphisms $f_j\in\aut(X_j), j=0,1,2$ that are symplectic. Huybrechts has proven \cite{Huy} that one has
    \[  (f_j)^\ast=\ide\colon\ \ \ A^2(X_j)\ \to\ A^2(X_j)\ \ \ (j=0,1,2)\ .\]
    Theorem \ref{main}, combined with the commutative diagram
     \[ \begin{array}[c]{ccc}
        A^2_{hom}(S) & \xrightarrow{ f^\ast} & A^2_{hom}(S) \\
        \uparrow{\scriptstyle (\Psi_j)^\ast}&&  \uparrow{\scriptstyle (\Psi_j)^\ast}\\
        A^2_{hom}(X_j) & \xrightarrow{ (f_j)^\ast} & A^2_{hom}(X_j) \\ 
        \end{array} \ \ \ \ \ \ (j=0,1,2) \]
        implies that
       \[ f^\ast=\ide\colon\ \ \ A^{2}_{hom}(S)\ \to\ A^{2}_{hom}(S)\ .\]   
       Since the $1$--dimensional subspace $A^2(S)^{+++}$ is fixed by $f$, this proves the corollary.
      \end{proof}

\subsection{Finite--dimensionality}

\begin{corollary}\label{main2} Let $S$ be a triple burger as in theorem \ref{main}, and let $X_j$ be the associated $K3$ surfaces. The morphism of Chow motives
   \[ (\Psi_0,\Psi_1,\Psi_2)\colon\ \ \ t(X_0)\oplus t(X_1)\oplus t(X_2)\ \xrightarrow{}\ t(S)\ \ \ \hbox{in}\ \MM_{\rm rat}\ \]
  is an isomorphism.  
   (Here, $t()$ denotes the transcendental part of the motive, as in theorem \ref{t2}.)
\end{corollary}

 \begin{proof} We may suppose $S$ and the $X_j$ are defined over some subfield $k\subset\C$ which is finitely generated over $\QQ$. To prove the isomorphism of motives, it suffices to prove there is an isomorphism
  \[ \bigl((\Psi_0)_\ast,(\Psi_1)_\ast,(\Psi_2)_\ast\bigr)\colon\ \ \  A^2_{hom}((X_0)_K)\oplus  A^2_{hom}((X_1)_K)  \oplus A^2_{hom}((X_2)_K)\ \xrightarrow{\cong}\ A^2_{hom}(S_K)\  \]
  for all function fields $K=k(Z)$ of varieties $Z$ defined over $k$ \cite[Lemma 1.1]{Huy2}. This is equivalent to proving claim \ref{claim} for the surface $S_K$. 
  Since $\C$ is a universal domain, one can choose an embedding $K\subset\C$. As is well--known (cf. \cite[Appendix to Lecture 1]{B}), this induces an injection
    \[  A^2(S_K)\ \hookrightarrow\ A^2(S_\C)\ ,\]
    and so claim \ref{claim} for $S_K$ follows from claim \ref{claim} for $S_\C$.
                   \end{proof}

 \begin{corollary}\label{cor1} Let $S$ be as in theorem \ref{main}, and assume
   \[ \dim H^2_{tr}(S)\le 7\ .\]
   Then $S$ has finite--dimensional motive (in the sense of Kimura \cite{Kim}). What's more, $S$ has motive of abelian type (in the sense of \cite{V3}).
   \end{corollary}
   
  \begin{proof} Let $X_0, X_1, X_2$ be the associated $K3$ surfaces. Recall (proposition \ref{Htr}) that there is an isomorphism
     \[ H^2_{tr}(S)\cong H^2_{tr}(X_0)\oplus H^2_{tr}(X_1)\oplus H^2_{tr}(X_2)\ .\]  
   The $X_j$ being $K3$ surfaces, the dimension of $H^2_{tr}(X_j)$ is at least $2$, and so the assumption on $ H^2_{tr}(S)$ implies that
   \[ \dim  H^2_{tr}(X_j)\le 3\ \ \ (j=0,1,2)\ .\]
   It follows from \cite{Ped} that the $X_j$ have finite-dimensional motive. In view of corollary \ref{main2}, 
   the motive $t(S)$ is isomorphic to $t(X_0)\oplus t(X_1)\oplus t(X_2)$, and so
   this implies the corollary.
   
   To see that $S$ has motive of abelian type, one remarks that the $K3$ surfaces $X_j$ either have a Shioda--Inose structure, or are rationally dominated by a Kummer surface \cite{ShI}, \cite{Mor0}. This implies that their motive is actually a submotive of the motive of an abelian surface.
   \end{proof} 
      
\begin{remark}\label{NB} In fairness, I hasten to add that I am not sure whether surfaces $S$ as in corollary \ref{cor1} exist. Indeed, one might naively expect that inside the families
  \[ \XX_j\ \to\ B\ \ \ (j=0,1,2) \]
  of $K3$ surfaces associated to the family $\Ss\to B$ (cf. notation \ref{fam}), $\rho$--maximal surfaces lie analytically dense (and so $\rho$--maximal triple $K3$ burgers would also be analytically dense). But to prove this, one would need to know a Torelli result for this type of $K3$ surfaces.
  
  For this reason, corollary \ref{cor1} is only a conditional result.
  \end{remark}

  \section{Open questions}
  \label{sopen}
  
  \begin{question} Can one prove Torelli type theorems for families of triple $K3$ burgers as in theorem \ref{main} ?
  As noted in remark \ref{NB}, this would have interesting consequences for the distribution of Picard numbers, and for the existence of certain finite--dimensional motives.
  \end{question}

  \begin{question}\label{all} Let $S$ be a triple $K3$ burger as in theorem \ref{main}. I wonder whether a stronger version of proposition \ref{int} might be true: is it the case that (as for $K3$ surfaces)
    \[ A^1(S)_\ZZ\cdot A^1(S)_\ZZ=\ZZ[\oo_S]\ \ \ \subset\ A^2(S)_\ZZ\ \ ?? \]
    On a related note, does $S$ have a multiplicative Chow--K\"unneth decomposition, in the sense of \cite{SV} ?
  \end{question}
   
  \begin{question} Let $S$ be a triple $K3$ burger as in theorem \ref{main}. Is it the case that (as for $K3$ surfaces) the second Chern class
    $ c_2(T_S) \in A^2(S)$ lies in the subgroup $\QQ[\oo_S]$ ?
  \end{question} 
  
  \begin{question} Let $X$ be a $K3$ surface, and let $F$ be a simple rigid vector bundle on $X$. Voisin has proven \cite[Theorem 1.9]{V21} that $c_2(F)\in A^2(X)$ lies in the subgroup $\QQ[\oo_X]$.
  Can one prove a similar statement for triple $K3$ burgers ?
  
(Presumably, Voisin's argument for $K3$ surfaces can be adapted to triple $K3$ burgers ? At least the ``dimension of orbit'' part goes through unchanged (proposition \ref{orbit}). However, Voisin's argument also involves Riemann--Roch calculations, which rely on having trivial canonical bundle. I have not pursued this.)  
    \end{question}

  \begin{question} Let $\pi\colon\Ss\to B$ be a family of surfaces (i.e., a smooth projective morphism with $2$--dimensional fibres). According to Deligne \cite{De}, there is a decomposition isomorphism
    \[   R \pi_\ast \QQ \cong \bigoplus_i R^i \pi_\ast\QQ[-i]\  \]
    in the derived category of sheaves of $\QQ$--vector spaces on $B$. 
    If the fibres of $\pi$ are $K3$ surfaces, then according to Voisin \cite{V13}, one can choose an isomorphism that becomes {\em multiplicative\/} after shrinking the base $B$. Can one do the same for a family of triple $K3$ burgers ?
    
   (This is closely related to the existence of a multiplicative Chow--K\"unneth decomposition, cf. \cite[Section 4]{V6}.) 
     \end{question} 
     
  \begin{question} What are the generic and maximal Picard numbers for the families of triple $K3$ burgers of theorem \ref{main} ?
  \end{question}

  \begin{question} Constructing quadruple $K3$ burgers (i.e., surfaces satisfying the $m=4$ analogon of definition \ref{def}) seems a daunting task. 
  
  (For example: if we suppose $S$ is a canonical surface of general type with $p_g=4$ and $K^2=5$, then we know \cite{Hor1} that $S$ is isomorphic to a quintic in $\PP^3$ with rational double points. Consider the involutions 
   \[ \begin{split} \sigma_0  [x_0:x_1:x_2:x_3]\ &=  [-x_0:x_1:x_2:x_3]  \ ,\\
                          \sigma_1  [x_0:x_1:x_2:x_3]\ &=  [x_0:-x_1:x_2:x_3]  \ ,\\
                         \sigma_2    [x_0:x_1:x_2:x_3]\ &=  [x_0:x_1:-x_2:x_3]  \ ,\\
                          \sigma_3    [x_0:x_1:x_2:x_3]\ &=  [x_0:x_1:x_2:-x_3]  \ ,\\
                          \end{split} \]
            If $S$ is a hypersurface invariant under $\sigma_j$ (i.e., the defining equation of $S$ contains only even powers of $x_j$), the quotient $S/<\sigma_j>$ is a $K3$ surface with double points. However, clearly there is no quintic hypersurface invariant under all $4$ involutions $\sigma_j$ !)
            
  The following is a weaker question: can one at least find general type surfaces $S$ with $p_g(S)=4$ such that the transcendental cohomology of $S$ splits in $4$ pieces of $K3$ type ? 
  And what about $p_g>4$ ?        
    \end{question}

\vskip1cm
\begin{nonumberingt} Thanks to the editor for his patience. Thanks to the referee for constructive remarks and for spotting an important oversight in a prior version.
Thanks to Kai and Len, for enjoying Sesamstraat as much as I do.
\end{nonumberingt}


\vskip1cm
\begin{thebibliography}{dlPG99}



\bibitem{BPV} W. Barth, K.Hulek, C. Peters and A. van de Ven, Compact complex surfaces, 2nd edition, Springer--verlag Berlin 2004,

\bibitem{BC} V. Batyrev and D. Cox, On the Hodge structure of projective hypersurfaces in toric varieties. Duke Math. J. 75 no. 2 (1994), 293---338,

\bibitem{BV} A. Beauville and C. Voisin, On the Chow ring of a $K3$ surface, J. Alg. Geom. 13 (2004), 417---426,

\bibitem{Beau3} A. Beauville, On the splitting of the Bloch--Beilinson filtration, in: Algebraic cycles and motives (J. Nagel and C. Peters, editors), London Math. Soc. Lecture Notes 344, Cambridge University Press 2007,

\bibitem{B} S. Bloch, Lectures on algebraic cycles, Duke Univ. Press Durham 1980,

\bibitem{B2} S. Bloch, Algebraic cycles and higher K--theory, Advances in Math. vol. 61 (1986), 267---304,

\bibitem{B3} S. Bloch, The moving lemma for higher Chow groups, J. Alg. Geom. 3 (1994), 537---568,



\bibitem{De} P. Deligne, Th\'eor\`eme de Lefschetz et crit\`eres de d\'eg\'enerescence de suites spectrales, Publ. Math. IHES 35 (1968), 107---126,

\bibitem{D} P. Deligne, La conjecture de Weil pour les surfaces $K3$, Invent. Math. 15 (1972), 206---226,

\bibitem{Del} C. Delorme, Espaces projectifs anisotropes, Bull. Soc. Math. France 103 (1975), 203---223,

\bibitem{Dol} I. Dolgachev, Weighted projective varieties, in: Group actions and vector fields, Vancouver 1981, Springer Lecture Notes in Mathematics 956, Springer Berlin Heidelberg New York 1982,

\bibitem{LFu} L. Fu, On the action of symplectic automorphisms on the $CH_0$--groups of some hyper-K\"ahler fourfolds, Math. Z. 280 (2015), 307---334,

\bibitem{F} W. Fulton, Intersection theory, Springer--Verlag Ergebnisse der Mathematik, Berlin Heidelberg New York Tokyo 1984,

\bibitem{Gar} A. Garbagnati, Smooth double covers of $K3$ surfaces, arXiv:1605.03438,

\bibitem{Hor1} E. Horikawa, On deformation of quintic surfaces, Invent. Math. 31 (1975), 43---85,

\bibitem{Hor0} E. Horikawa, Algebraic surfaces of general type with small $c_1^2$, I, Ann. Math. 104 (1976), 357---387,

\bibitem{Hor} E. Horikawa, Algebraic surfaces of general type with small $c_1^2$, II, Invent. Math. 37 no. 2 (1976), 121---155,

\bibitem{Hor3} E. Horikawa, Algebraic surfaces of general type with small $c_1^2$, III, Invent. Math. 47 (1978), 209---248,

\bibitem{Hor4} E. Horikawa, Algebraic surfaces of general type with small $c_1^2$, IV, Invent. Math. 50 (1979), 103---128,

\bibitem{Huy} D. Huybrechts, Symplectic automorphisms of K3 surfaces of arbitrary finite order, Math. Res. Lett. 19 no. 4 (2012), 947---951,

\bibitem{Huyb} D. Huybrechts, Lectures on $K3$ surfaces, Cambridge University Press 2016,

\bibitem{Huy2} D. Huybrechts, Motives of derived equivalent $K3$ surfaces, Abh. Math. Sem. Hamburg,

\bibitem{Ili} V. Iliev, A note on certain surfaces, Bull. London Math. Soc. 16 (1984), 135---138,

\bibitem{J2} U. Jannsen, Motivic sheaves and filtrations on Chow groups, in: Motives (U. Jannsen et alii, eds.), Proceedings of Symposia in Pure Mathematics Vol. 55 (1994), Part 1,  

\bibitem{J4} U. Jannsen, On finite--dimensional motives and Murre's conjecture, in: Algebraic cycles and motives (J. Nagel and C. Peters, editors), Cambridge University Press, Cambridge 2007,

\bibitem{KMP} B. Kahn, J. Murre and C. Pedrini, On the transcendental part of the motive of a surface, in: Algebraic cycles and motives (J. Nagel and C. Peters, editors), Cambridge University Press, Cambridge 2007,

\bibitem{Kim} S. Kimura, Chow groups are finite dimensional, in some sense,
Math. Ann. 331 (2005), 173---201,

\bibitem{Kun} V. Kynev, An example of a simply--connected surface of general type for which the local Torelli theorem does not hold, C. R. Acad. Bulgare Sci. 30 no. 3 (1977), 323---325,

\bibitem{tod} R. Laterveer, Algebraic cycles and Todorov surfaces, to appear in Kyoto Journal of Math., arXiv:1609.09629,

\bibitem{hor} R. Laterveer, Algebraic cycles and special Horikawa surfaces, preprint,

\bibitem{Lev} M. Levine, Techniques of localization in the theory of algebraic cycles, J. Alg. Geom. 10
(2001), 299---363,

\bibitem{Mor0} D. Morrison, On $K3$ surfaces with large Picard number, Invent. Math. 75 No 1 (1984), 105---121,

\bibitem{Mor} D. Morrison, On the moduli of Todorov surfaces, in: Algebraic Geometry and Commutative Algebra in Honor of Masayoshi Nagata (H. Hijikata et al., eds.), vol. 1, Kinokuniya, Tokyo 1988,

\bibitem{Mur} J. Murre, On a conjectural filtration on the Chow groups of an algebraic variety, parts I and II, Indag. Math. 4 (1993), 177---201,

\bibitem{MNP} J. Murre, J. Nagel and C. Peters, Lectures on the theory of pure motives, Amer. Math. Soc. University Lecture Series 61, Providence 2013,

\bibitem{OG} K. O'Grady, Moduli of sheaves and the Chow groups of $K3$ surfaces, Journal de Math. Pures et Appliqu\'ees 100 no. 5 (2013), 701---718,




\bibitem{PZ} G. Pearlstein and Z. Zhang, A generic global Torelli theorem for certain Horikawa surfaces, arXiv:1702.06204v1,

\bibitem{Ped} C. Pedrini, On the finite dimensionality of a $K3$ surface, Manuscripta Mathematica 138 (2012), 59---72,

\bibitem{PS} C. Peters and J. Steenbrink, Mixed Hodge structures, Springer--Verlag Ergebnisse der Mathematik, Berlin Heidelberg New York 2008,

\bibitem{Reid} M. Reid, Surfaces with $p_g=3$, $K^2=4$ according to E. Horikawa and D. Dicks, available from homepages.warwick.ac.uk/~masda/surf/more/Dicks.pdf,

\bibitem{Rito} C. Rito, A note on Todorov surfaces, Osaka J. Math. 46 no. 3 (2009), 685---693,

\bibitem{Roi} A. Rojtman, The torsion of the group of 0--cycles modulo rational equivalence, Annals of Mathematics 111 (1980), 553---569,

\bibitem{Sc} T. Scholl, Classical motives, in: Motives (U. Jannsen et alii, eds.), Proceedings of Symposia in Pure Mathematics Vol. 55 (1994), Part 1, 

\bibitem{SV} M. Shen and C. Vial, The Fourier transform for certain hyperK\"ahler fourfolds, Memoirs of the AMS 240 (2016), no.1139,


\bibitem{ShI} T. Shioda and H. Inose, On singular $K3$ surfaces, in: Complex Analysis and Algebraic Geometry, Iwanami Shouten, Tokyo 1977, 119---136,

\bibitem{Ste} J. Steenbrink, On the Picard group of certain smooth surfaces in weighted projective spaces, in: Algebraic Geometry, Proceedings, La R\'abida 1981, Springer Lecture Notes in Mathematics vol. 961, Springer Berlin 1982, 302---313,

\bibitem{Tod} A. Todorov, Surfaces of general type with $p_g= 1$ and $(K,K) = 1$, Ann. Sci. de l'Ecole Normale Sup. 13 (1980), 1---21,

\bibitem{T} B. Totaro, Chow groups, Chow cohomology, and linear varieties, Forum of Mathematics, Sigma (2014), vol. 1, e1,



\bibitem{V3} C. Vial, Remarks on motives of abelian type, Tohoku Math. J. 69 (2017), 195---220,

\bibitem{V6} C. Vial, On the motive of some hyperk\"ahler varieties, J. f\"ur Reine u. Angew. Math. 725 (2017), 235---247,


\bibitem{Voe} V. Voevodsky, A nilpotence theorem for cycles algebraically equivalent to zero, Internat. Math. Research Notices 4 (1995), 187---198,

\bibitem{V9} C. Voisin, Remarks on zero--cycles of self--products of varieties, in: Moduli of vector bundles, Proceedings of the Taniguchi Congress  (M. Maruyama,  ed.), Marcel Dekker New York Basel Hong Kong 1994,





\bibitem{V13} C. Voisin, Chow rings and decomposition theorems for $K3$ surfaces and Calabi--Yau hypersurfaces, Geom. Topol.
16 (2012), 433---473,


\bibitem{V0} C. Voisin, The generalized Hodge and Bloch conjectures are equivalent for general complete intersections, Ann. Sci. Ecole Norm. Sup. 46, fascicule 3 (2013), 449---475,

\bibitem{V8} C. Voisin, Bloch's conjecture for Catanese and Barlow surfaces, J. Differential Geometry 97 (2014), 149---175,

\bibitem{V21} C. Voisin, Rational equivalence of $0$--cycles on $K3$ surfaces and conjectures of Huybrechts and O'Grady, in: Recent advances in algebraic geometry (C. Hacon et al., eds.), Cambridge University Press 2014,

\bibitem{V1} C. Voisin, The generalized Hodge and Bloch conjectures are equivalent for general complete intersections, II, J. Math. Sci. Univ. Tokyo  22 (2015), 491---517,

\bibitem{Vo} C. Voisin, Chow Rings, Decomposition of the Diagonal, and the Topology of Families, Princeton University Press, Princeton and Oxford, 2014,



\bibitem{Zh} D.--Q. Zhang, Algebraic surfaces with log--canonical singularities and the fundamental groups of their smooth parts, Transactions of the AMS 348 No. 10 (1996), 4175---4184.


\end{thebibliography}
\end{document}